\documentclass{amsart}
\usepackage{amsmath,amssymb,mathrsfs,bbm}
\usepackage{amsthm}

\usepackage{graphicx}
\date{}
\def\im{\qopname\relax o{Im}}
\def\saf{\qopname\relax o{SAF}}
\def\re{\qopname\relax o{Re}}
\def\rank{\qopname\relax o{rank}}

\long\def\remove#1{{\color{red}#1}}
\long\def\newtext#1{{\color{blue}#1}}
\long\def\remove#1{}
\long\def\newtext#1{#1}

\theoremstyle{theorem}
\newtheorem{theo}{Theorem}
\newtheorem{lemm}{Lemma}
\newtheorem{prop}{Proposition}
\newtheorem{conj}{Conjecture}
\newtheorem{coro}{Corollary}

\theoremstyle{remark}
\newtheorem{rema}{Remark}
\newtheorem{exam}{Example}

\theoremstyle{definition}
\newtheorem{defi}{Definition}

\author {Ivan Dynnikov and Alexandra Skripchenko}
\thanks{The work is supported by the Russian Science Foundation under grant 14-50-00005 and performed in Steklov Mathematical Institute of Russian Academy of Sciences, Moscow, Russia.}
\address{\noindent V.\,A.\,Steklov Mathematical Institute of Russian Academy of Science, 8 Gubkina Str., Moscow 119991, Russia}
\email{dynnikov@mech.math.msu.su}
\email{sashaskrip@gmail.com}

\title[{}]{Minimality of interval exchange transformations with restrictions}

\begin{document}
\maketitle

\begin{abstract}
It is known since 40 years old paper by M.\,Keane that minimality is a generic (i.e.\ holding with probability one) property
of an irreducible interval exchange transformation. If one puts some integral linear restrictions on
the parameters of the interval exchange transformation, then minimality may become
an ``exotic'' property. We conjecture in this paper that this occurs
if and only if the linear restrictions contain a Lagrangian subspace
of the first homology of the suspension surface. We partially prove it
in the `only if' direction and provide a series of examples to support
the converse one. We show that the unique ergodicity
remains a generic property if the restrictions on the parameters do not contain
a Lagrangian subspace (this result is due to Barak Weiss).
\end{abstract}

\section{Introduction}
Interval exchange transformations are maps from an interval to itself that look like suggested by
the name: the interval is cut into several subintervals that are ``glued back'' in a different way.
On each subinterval the map has the form of a translation $x\mapsto x+\mathrm{const}$.

The first examples of interval exchange
transformations were studied by V.\,Arnold (see, e.g.\ \cite{a63}) and A.\,Katok and A.\,Stepin \cite{ks67} in early 1960s.
The general notion was introduced by V.\,Oseledets in \cite{o68}.

\begin{defi}
An interval exchange transformation is specified by a natural number $n$, a permutation $\pi\in S_n$,
and a probability vector $\mathbf a=(a_{1}, a_{2},\cdots,a_{n})$, $\sum_ia_i=1$, of lengths of the subintervals. The transformation
$T_{\pi,\mathbf a}$ determined by this data is the map $[0,1)\rightarrow[0,1)$ defined by the formula:
$$T_{\pi,\mathbf a}(x)=x-x_i+\widetilde x_{\pi^{-1}(i)}\quad\text{if}\quad x\in[x_{i-1},x_i),$$
where
$$x_0=\widetilde x_0=0,\quad x_i=\sum_{j=1}^ia_j,\quad\widetilde x_i=\sum_{j=1}^ia_{\pi(j)},\quad i=1,\ldots,n.$$
\end{defi}

We will only be interested in interval exchange transformations $T_{\pi,\mathbf a}$ with
an \emph{irreducible} permutation $\pi\in S_n$, i.e.\ such that no subset of the form $\{1,2,\ldots,k\}$ with $1\leqslant k<n$ is invariant
under $\pi$.

Any interval exchange transformation is invertible with the inverse given by $T_{\pi,\mathbf a}^{-1}=T_{\pi^{-1},\mathbf a\cdot\pi}$,
where $\mathbf a\cdot\pi=(a_{\pi(1)},\ldots,a_{\pi(n)})$.

\begin{defi}
Let $T$ be an interval exchange transformation. \emph{The $T$-orbit} of a point $x\in[0,1)$ is the subset
$\{T^k(x)\;;\;k\in\mathbb Z\}$.
\end{defi}

Interval exchange transformations appear naturally as the first return
map on a transversal for singular foliations defined by a closed $1$-form on a surface.
More precisely, for each interval exchange transformation
one can construct a translation surface
such that the transformation
will be the first return map of the vertical foliation on some horizontal
interval (see~\cite{m06,vi06,z06} for details); and vice versa, for each foliation on an oriented surface
that is defined by a closed one-form the first return map on any transversal
is an interval exchange transformation.

\begin{defi}
An interval exchange transformation $T$ is called \emph{minimal} if all $T$-orbits are everywhere dense in $[0,1)$.
\end{defi}

The question of finding conditions under which an interval exchange transformation
is minimal was posed by M.\,Keane in \cite{k75}.
He shows that all interval exchange transformations with irreducible permutation and
rationally independent lengths $a_i$ satisfy the following condition, which, in turn, implies that
the transformation is minimal.

\begin{defi}
We say that an interval exchange transformation $T_{\pi,\mathbf a}$ \emph{satisfies Keane's condition} if
the $T_{\pi,\mathbf a}$-orbits of the points
$x_1=a_1$, $x_2=a_1+a_2$, \ldots, $x_{n-1}=\sum_{i=1}^{n-1}a_i$ are
pairwise disjoint and infinite.
\end{defi}

Moreover, Keane posed a conjecture in \cite{k75}, which was later proved by H.\,Masur~\cite{m82} and W.\,Veech~\cite{v82},
that almost all irreducible interval exchange transformations are uniquely ergodic.

There are, however, several ways how families of irreducible non-generic interval exchange transformations whose parameters
are dependent over $\mathbb Q$ may arise, in which case the results mentioned above
do not apply. In such a family, almost all interval exchange transformations may still
be minimal (and even uniquely ergodic), but it may also happen that all or almost all are not. We illustrate this with two simple examples.

\begin{exam}\label{examplenovak}
Let $n=2k$ and
\begin{equation}\label{permnov}
\pi=\begin{pmatrix}1&2&3&4&\ldots&2k-3&2k-2&2k-1&2k\\
4&3&6&5&\ldots&2k&2k-1&2&1\end{pmatrix}\in S_{n},
\end{equation}
$\mathbf a=(a_1,a_2,a_1,a_2,\ldots,a_1,a_2)\in\mathbb R^{n}$.
If $a_1$ and $a_2$ are incommensurable, then $T_{\pi,\mathbf a}$ is minimal and
uniquely ergodic.
In the particular case $k=2$ such interval exchange transformations appear in \cite{n10}.
\end{exam}

\begin{exam}
Let $\pi=\begin{pmatrix}1&2&3&4\\2&4&3&1\end{pmatrix}$ and $a_1=a_4$. Then for any
$x\in[x_2,x_3)$ we have $T_{\pi,\mathbf a}(x)=x$, so, the transformation is \emph{not} minimal
\newtext{unless $a_3=0$, i.e.\ $x_2=x_3$}.
\end{exam}

In general, the situation can be more complicated. Let $\mathscr U$ be a subspace of $\mathbb R^n$
defined by a homogeneous system of linear equations with integral coefficients, and $\pi\in S_n$ a fixed
permutation. We denote by $\Delta^{n-1}$ the simplex $\{\mathbf a\in\mathbb R^n\;;\;a_i\geqslant0,\ \sum_ia_i=1\}$.
The subset  
$$M(\pi,\mathscr U)=\{\mathbf a\in\Delta^{n-1}\cap\mathscr U\;;\;T_{\pi,\mathbf a}\text{ is minimal}\}$$
can be a nontrivial
part of $\Delta^{n-1}\cap\mathscr U$ as the following two examples show.

\begin{exam}\label{exam-minimal}
Let $\pi$ be as in the previous example and $\mathbf a$ satisfy the following restriction: $3a_1=a_2+a_4$. Then:
\begin{enumerate}
\item
if $a_2>a_1$ then for any $x\in[\max(0,a_1-a_4),\min(a_1,a_2-a_1))$ we have $T_{\pi,\mathbf a}^4(x)=x$ (this is a
straightforward check), and $T_{\pi,\mathbf a}$ is not minimal;
\item
if $a_2<a_1$ and $a_1,a_2,a_3$ are linearly independent over $\mathbb Q$, then $T_{\pi,\mathbf a}$ is minimal (this will
be shown in Subsection~\ref{asymptotic-n-separating}, Example~\ref{exam-seapr}),
\end{enumerate}
see Fig.~\ref{gasket1} on the left.
\end{exam}

\begin{exam}\label{rauzy}
Let $\pi=\begin{pmatrix}1&2&3&4&5&6&7\\3&6&5&2&7&4&1\end{pmatrix}$ and
$\mathbf a=(a_1,a_2,a_3,a_3,a_1,a_1,a_2)$, $3a_1+2a_2+2a_3=1$. Then the set of points $(a_1,a_2)$
for which the interval exchange transformation $T_{\pi,\mathbf a}$ is minimal forms a fractal
set shown in Fig.~\ref{gasket1} on the right.
\begin{figure}[ht]
\centerline{\includegraphics{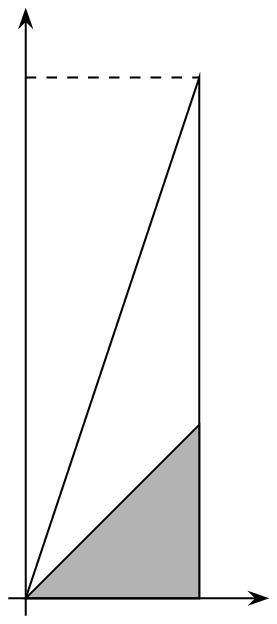}\put(-90,0){$0$}\put(-18,0){$a_1$}\put(-90,158){$\frac34$}\put(-33,0){$\frac14$}%
\put(-95,175){$a_2$}\put(-63,17){$\scriptstyle M(\pi,\mathscr U)$}\put(-25,60){$\Delta^3\cap\mathscr U$}\hskip2cm
\includegraphics{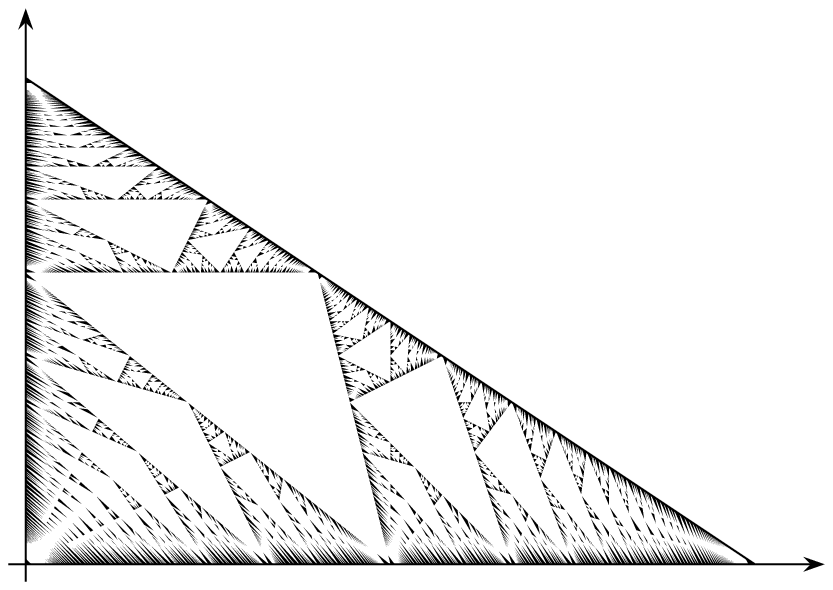}\put(-18,0){$a_2$}\put(-255,165){$a_1$}\put(-250,0){$0$}\put(-33,0){$\frac12$}\put(-250,148){$\frac13$}\put(-155,100){$\Delta^6\cap\mathscr U$}}
\caption{The subset $M(\pi,\mathscr U)\subset\Delta^n\cap\mathscr U$ in Examples 3 (left) and 4 (right).
Only points with $\dim_{\mathbb Q}\langle a_1,a_2,1\rangle=3$ are considered.}\label{gasket1}
\end{figure}
Up to a projective transformation this set is the so called
Rauzy gasket~\cite{l93,dd08,as11}. As shown in \cite{ahs13} it has Hausdorff dimension between $1$ and $2$.

If, in this example, we take the parameter vector $\mathbf a$ of the form $\mathbf a=(a_1,a_2,a_3,a_3-e,a_1,a_1-e,a_2+e)$ with $e\ne0$,
$-a_2<e<\min(a_1,a_3)$,
then the transformation $T_{\pi,\mathbf a}$ will never be minimal. This will be discussed in more detail in Subsection~\ref{spi}.
\end{exam}

These two examples demonstrate two possible behaviors of the set minimal interval exchange transformation
with integral linear restrictions on parameters. We refer to those behaviors as \emph{stable} (as in Example 3)
and \emph{unstable} (as in Example 4).

We introduce below an easily checkable condition for a set of integral
linear restrictions, and conjecture that the condition is exactly what is responsible for instability
of minimal interval exchange transformations with integral linear restrictions on the parameters.
The set of restrictions satisfying the condition
is termed \emph{rich} and otherwise \emph{poor}. We show that minimal and uniquely ergodic
interval exchange transformations that are subject to a poor set of
restrictions are always stable.

We also demonstrate several natural sources of families of interval exchange transformations
with integral linear restrictions on parameters. These sources include:
\begin{enumerate}
\item
\newtext{rank two interval exchange transformations};
\item
\remove{measured singular foliations on surfaces (orientable or not)}
\newtext{singular measured foliations on Riemann surfaces defined
by a quadratic differential};
\item
interval exchange transformations with flips;
\item
\newtext{measured foliations on non-orientable surfaces};
\item
Novikov's problem on plane sections of triply periodic surfaces;
\item
systems of partial isometries;
\item
interval translation mappings.
\end{enumerate}

In several cases, the set of restrictions appears to be rich. In most such cases, our conjecture
about instability of minimal interval exchange transformations
translates into a statement that is either known or has been conjectured to be true.

M.\,Boshernitzan in \cite{b88} considered interval exchange transformations $T_{\pi,\mathbf a}$ with
$\dim_{\mathbb Q}\langle a_1,\ldots,a_n\rangle=2$ \newtext{(rank two interval exchange transormations)}
and proposed a way to test such transformations for minimality.
\remove{However, his test applies to individual maps not saying much about
the whole family of minimal transformations within a two-parameter family defined
by integral linear restrictions.}
\newtext{A family of rank two interval exchange transformations having the same discrete pattern,
in the terminology of~\cite{b88}, is an instance of a family of transformations satisfying
a fixed set of integral linear restrictions, in our terminology. Theorem~9.1 of~\cite{b88}
states that minimality is a stable property in such a family. We show in Subsection~\ref{ranktwo}
that a rank two interval exchange transformation can be minimal only
if its set of restrictions is poor.}
\remove{ A similar test for minimality was suggested by E.\,Lanneau and C.\,Boissy
for so called linear involutions \cite{bl09}. }

The paper is organized as follows. In Section~\ref{criteriasec} we introduce all basic notions and define
poor and rich restriction spaces as well as the notion of stability for minimal interval exchange transformations
with restrictions.

In Section~\ref{poor->stable} we consider interval exchange transformations with poor restrictions.
We show that, in this case, minimal uniquely ergodic ones are always stably minimal.
We also show that in a small neighborhood of such a transformation almost
all transformations satisfying the same family of restrictions are uniquely ergodic.

In Section~\ref{instability} we conjecture that minimal interval exchange transformations
with rich restrictions are never stably minimal and, moreover, that minimality
occurs with probability $0$ in this case.

In Section~\ref{Examples} we discuss how interval exchange transformations with
rich or poor restrictions arise in different subjects. We show that some known problems are
particular instances of our conjectures.

\subsection*{Acknowledgement}We are indebted to Barak Weiss who communicated to us a proof of Theorem~\ref{uniquelyergodictheo}.
We also thankful to our anonymous referees for many valuable remarks and suggestions. In particular, the first referee drew our attention
to the connection between our subject and the SAF invariant, and the second one pointed out
a misinterpretation of Boshernitzan's result~\cite{b88} in an earlier version of the paper.

\section{Preliminaries}\label{criteriasec}

Throughout Sections~\ref{criteriasec}--\ref{instability} we assume that $\pi$ is a fixed irreducible permutation $\pi\in S_n$.

\subsection{Suspension surface}\label{prelimsec}
Interval exchange transformations (see \newtext{the} Introduction for the definition)
are intimately related with singular foliations that can be
defined by a closed $1$-form on an oriented surface. The leaves of such a foliation can also
be considered as trajectories of a Hamiltonian system on the surface. The
corresponding interval exchange transformation appears as the first return map
for some transversal of the foliation (or flow). The surface and the $1$-form are constructed from
an interval exchange transformation as follows.

\begin{defi}
Let $\mathbf a\in\Delta^{n-1}$. Take the square $[0,1]\times[0,1]$ and make the following identifications:
\begin{equation}\label{identifications}\begin{aligned}
(x,1)&\sim(T_{\pi,\mathbf a}(x),0),\ &x\in[0,1);\\
(x,1)&\sim(\lim_{y\rightarrow x-0}T_{\pi,\mathbf a}(y),0),\ &x\in(0,1];\\
(0,y)&\sim(0,0),\ &y\in[0,1];\\
(1,y)&\sim(1,0),\ &y\in[0,1].
\end{aligned}
\end{equation}
The obtained surface will be denoted $\Sigma_{\pi,\mathbf a}$ and called \emph{the suspension surface of $T_{\pi,\mathbf a}$}.
A smooth structure can be chosen on $\Sigma_{\pi,\mathbf a}$ so
that $dx$ becomes a smooth differential $1$-form on $\Sigma_{\pi,\mathbf a}$, which we denote by $\omega_{\pi,\mathbf a}$ or simply $\omega$ if $\mathbf
a$ and $\pi$ are fixed.
\end{defi}

\begin{rema}
From topological point of view, this construction is equivalent to the one
in~\cite[page 174]{m82} or~\cite[Section~12]{vi06}, where the surface comes with more structure, namely, a specific Riemann surface structure
and a holomorphic $1$-form of which $\omega$ is the real part.
\end{rema}

The fibration of $[0,1]\times[0,1]$ by vertical arcs $\{x\}\times[0,1]$ becomes,
after the identifications, a singular foliation on $\Sigma_{\pi,\mathbf a}$, which we
denote by $\mathcal F_{\pi,\mathbf a}$. Its leaves are locally defined
by the equation $\omega=0$. The foliation $\mathcal F_{\pi,\mathbf a}$
has finitely many singularities, which may occur only at the points of the subset
$$S_{\pi,\mathbf a}=\{(x_1,1),\ldots,(x_{n-1},1)\}/{\sim}\,\subset\Sigma_{\pi,\mathbf a}.$$
All singularities have the form
of a saddle or a multiple saddle, see Fig.~\ref{saddles}.
\begin{figure}[ht]
\centerline{\includegraphics[scale=0.7]{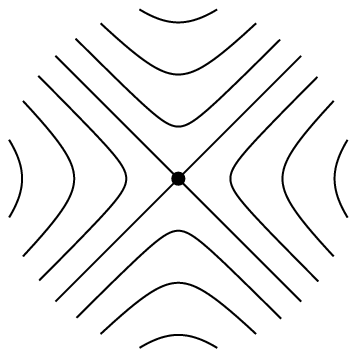}\hskip1cm\includegraphics[scale=0.7]{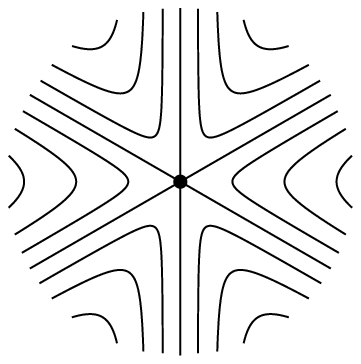}\hskip1cm\includegraphics[scale=0.7]{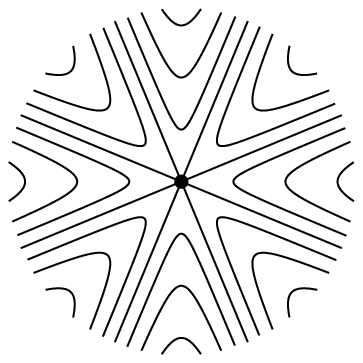}}
\caption{Singularities of $\mathcal F_{\pi,\mathbf a}$}\label{saddles}
\end{figure}

The points of $S_{\pi,\mathbf a}$ that are regular for $\mathcal F_{\pi,\mathbf a}$ are called \emph{marked points}.

One can see that, under an appropriate
orientation of leaves, the first return map of the foliation $\mathcal F_{\pi,\mathbf a}$ on the parametrized transversal $\gamma(t)=(t,1/2)/{\sim}\in
\Sigma_{\pi,\mathbf a}$
coincides with $T_{\pi,\mathbf a}$ except at finitely many points of discontinuity of $T_{\pi,\mathbf a}$, where the first return map is not well defined.
So, minimality of $T_{\pi,\mathbf a}$ is equivalent to that of $\mathcal F_{\pi,\mathbf a}$, where minimality is understood in a weak sense:
$\mathcal F_{\pi,\mathbf a}$ is minimal if every regular leaf is everywhere dense.

Keane's condition for $T_{\pi,\mathbf a}$ means exactly that
every leaf of $\mathcal F_{\pi,\mathbf a}$ is simply connected and contains at most one
marked point or singularity of $\mathcal F_{\pi,\mathbf a}$. In particular, there are no saddle connections.

Note that we consider a singular point
of $\mathcal F_{\pi,\mathbf a}$ and separatrices coming from it
as parts of a single singular leaf of $\mathcal F_{\pi,\mathbf a}$.

Let $\Sigma$ be a closed oriented surface of genus $g\geqslant1$, $S$ a finite subset of $\Sigma$,
and $m:S\rightarrow\mathbb N\cup\{0\}$ a function such that $\sum_{P\in S}m(P)=2g-2$.
The construction above gives a local parametrization of the moduli space $\mathcal M(\Sigma,S,m)$
of closed $1$-forms $\omega$ on a surface $\Sigma$ with
the following properties:
\begin{enumerate}
\item
$\omega$ is the real part of a holomorphic $1$-form on $\Sigma$
for some complex structure;
\item
$\omega$ has no zeros outside $S$;
\item
every point $P\in S$ has a neighborhood in which $\omega$ can be written as $\mathrm{Re}\bigl(d(x+\mathbbm iy)^{m(P)+1}\bigr)$
for some local coordinates $x,y$.
\end{enumerate}
Two forms $\omega$ and $\omega'$ are considered equivalent if there is a diffeomorphism $\varphi:(\Sigma,S)\rightarrow
(\Sigma,S)$ isotopic to the identity modulo $S$ such that $\omega=\mathrm{const}\cdot\varphi^*\omega'$ with $\mathrm{const}>0$.

In order to get a local parametrization of $\mathcal M(\Sigma,S,m)$ around a point represented by a $1$-form $\omega$
one needs to find a simple transversal arc (or loop) $\gamma$ of the induced foliation with endpoints in $S$ such that~$\gamma$
intersects all leaves and all saddle connections if any. One parametrizes $\gamma$ so that
$dt=\mathrm{const}\cdot\omega|_\gamma$ with $\mathrm{const}>0$, where $t$ is the parameter.
Then the first return map of the induced foliation on $\gamma$ (extended to the ambiguity points to
be continuous on the right) is an interval exchange map $T_{\rho,\mathbf a}$
such that $(\Sigma_{\rho,\mathbf a},S_{\rho,\mathbf a},\omega_{\rho,\mathbf a})$ can be identified with $(\Sigma,S,\omega)$.
A perturbation of $\omega$ such that $\gamma$ remains a transversal leaves the permutation $\rho$
unchanged, so, the components of $\mathbf a$ become local coordinates in $\mathcal M(\Sigma,S,m)$
around the point represented by $\omega$. These parameters are also interpreted as
coordinates of the relative cohomology class of $\omega$ in $H^1(\Sigma,S;\mathbb R)$.

For more details about suspension surfaces of interval exchange transformations
we refer the reader to~\cite{vi06}.

\subsection{Restrictions}\label{restrictionsec}

\begin{defi}
\emph{An integral linear restriction} (or simply \emph{a restriction}) is a linear function $r\in(\mathbb R^n)^*$
with integral coefficients. An interval exchange transformation $T_{\pi,\mathbf a}$
is said to \emph{satisfy a restriction $r$} if $r(\mathbf a)=0$.
A real subspace $\mathscr R\subset(\mathbb R^n)^*$
spanned
by a family of integral restrictions satisfied by $T_{\pi,\mathbf a}$ is called \emph{a restriction space of $T_{\pi,\mathbf a}$}.
The subspace generated by all such integral restrictions
is called \emph{the full restriction space of~$T_{\pi,\mathbf a}$} and denoted by $\mathscr R(T_{\pi,\mathbf a})$.
\end{defi}

The reader is alerted that the normalization condition $\sum_ia_i=1$ in the definition
of an interval exchange transformation is introduced only to eliminate the rescaling degree of freedom,
which is redundant,
and has nothing to do with integral linear restrictions.

We will use the interpretation of integral linear restrictions
in terms of the suspension surface $\Sigma_{\pi,\mathbf a}$ and the $1$-form
$\omega_{\pi,\mathbf a}$. Denote by $\gamma_1,\ldots,\gamma_n$ the images of the oriented intervals
$[0,x_1]$, $[x_1,x_2]$, \ldots, $[x_{n-1},1]$, respectively, in $\Sigma_{\pi,\mathbf a}$
under identifications~\eqref{identifications}.
They are regarded as $1$-chains, and their homology classes $h_i=[\gamma_i]$
form a basis in the relative homology group $H_1(\Sigma_{\pi,\mathbf a},S_{\pi,\mathbf a};\mathbb Z)$, since
$\gamma_1$, \ldots, $\gamma_n$ form the $1$-skeleton of a cell decomposition of $\Sigma_{\pi,\mathbf a}$
with $0$-skeleton $S_{\pi,\mathbf a}$. We obviously have
$$a_i=\int_{h_i}\omega_{\pi,\mathbf a}.$$
Thus all possible restrictions for $T_{\pi,\mathbf a}$ can be regarded as elements of the relative integral homology
group $H_1(\Sigma_{\pi,\mathbf a},S_{\pi,\mathbf a};\mathbb Z)$. Namely, the restriction corresponding
to a homology class $r\in H_1(\Sigma_{\pi,\mathbf a},S_{\pi,\mathbf a};\mathbb Z)$ has the form
$$\int_r\omega_{\pi,\mathbf a}=0.$$

The set of all restrictions satisfied by $T_{\pi,\mathbf a}$ is then a subgroup of $H_1(\Sigma_{\pi,\mathbf a},S_{\pi,\mathbf a};\mathbb Z)$,
and the full restriction space $\mathscr R(T_{\pi,\mathbf a})$ is a subspace of
$H_1(\Sigma_{\pi,\mathbf a},S_{\pi,\mathbf a};\mathbb R)=H_1(\Sigma_{\pi,\mathbf a},S_{\pi,\mathbf a};\mathbb Z)\otimes\mathbb R$.

\begin{defi}\label{realized}
A restriction $r\in\mathscr R(T_{\pi,\mathbf a})$ will be called \emph{an arc restriction} if it can be presented
by a single arc with distinct endpoints (which must be from $S_{\pi,\mathbf a}$). A restriction $r\in\mathscr R(T_{\pi,\mathbf a})$ will be called
\emph{a loop restriction} if it can be presented by
a simple closed curve.
If $r\in\mathscr R(T_{\pi,\mathbf a})$ is a loop restriction
or an arc restriction we say that $r$ is \emph{a simple restriction}.

We say that a simple restriction $r$ is \emph{realized by $T_{\pi,\mathbf a}$} if it can
be presented by a simple piecewise smooth arc or a simple closed piecewise smooth curve contained in a leaf (typically, singular leaf)
of $\mathcal F_{\pi,\mathbf a}$. In this case, the corresponding arc or closed curve will
be called \emph{a realization of $r$}. The number of intersections of this curve or arc
with the transversal $(0,1)\times\{1/2\}/{\sim}$ will be called \emph{the length} of the
realization. A realization $\rho$ of $r$ is \emph{elementary} if there is no decompostion $r=r_1+r_2$, $\rho=\rho_1\cup\rho_2$
with $r_1$ and $r_2$ simple restrictions realized by $\rho_1$ and $\rho_2$, respecitvely.
\end{defi}

Realizing a simple restriction is a stable property in the following sense.

\begin{prop}\label{zorich-general}
Suppose $T_{\pi,\mathbf a}$ realizes a simple restriction $r$. Then there exists an open neighborhood
$U\subset\Delta^{n-1}$ of $\mathbf a$ such that $T_{\pi,\mathbf a'}$ also realizes $r$ whenever $\mathbf a'\in U$ and
$\mathscr R(T_{\pi,\mathbf a'})\supset\mathscr R(T_{\pi,\mathbf a})$.
\end{prop}

This statement generalizes one of the arguments of~\cite{z84}.

\begin{proof}When $\mathbf a'$ varies within a small neighborhood of $\mathbf a$ the corresponding surface $\Sigma_{\pi,\mathbf a'}$
can be identified with $\Sigma_{\pi,\mathbf a}$ so that $S_{\pi,\mathbf a'}$ coincides with $S_{\pi,\mathbf a}$ and
the $1$-form $\omega_{\pi,\mathbf a'}$ remains close to $\omega_{\pi,\mathbf a}$.

Let $r=r_1+\ldots+r_k$ be a decomposition of $r$ into restrictions that admit an elementary realization, which clearly always exists.
Let $\rho_1,\ldots,\rho_k$ be elementary realizations of $r_1,\ldots,r_k$, respectively. Each~$\rho_i$ is either a regular closed leaf
(this may occur only if $k=1$, $r=r_1$) or a saddle connection. Saddle connections and closed leaves
persist under small perturbations if the integral of the $1$-form over them remains zero. The latter does
occur for $\omega_{\pi,\mathbf a'}$ if $\mathscr R(T_{\pi,\mathbf a'})\supset\mathscr R(T_{\pi,\mathbf a})$
since $r_i\in\mathscr R(T_{\pi,\mathbf a})$, $i=1,\ldots,k$.
\end{proof}

\subsection{Rich and poor restriction spaces}\label{rich-and-poor}

Let $\mathscr V$ be a finite dimensional vector space over $\mathbb R$ or $\mathbb Q$ equipped with a skew-symmetric
bilinear form $\langle\,,\rangle$.
Recall that a subspace $\mathscr W\subset\mathscr V$ is called \emph{isotropic}
if the restriction of $\langle\,,\rangle$ vanishes on $\mathscr W$.
A subspace $\mathscr W\subset\mathscr V$ is called \emph{coisotropic} if
$\langle u,v\rangle=0$ $\forall v\in\mathscr W$ implies $u\in\mathscr W$.

If $\mathscr V$ is symplectic, i.e.\ the form $\langle\,,\rangle$ is non-degenerate,
then any maximal isotropic subspace $\mathscr W\subset\mathscr V$
is called \emph{a Lagrangian subspace} of $\mathscr V$. The dimension
of a Lagrangian subspace is always one half of~$\dim\mathscr V$.
A subspace $\mathscr W\subset\mathscr V$ of a symplectic space is coisotropic if
and only if $\mathscr W$
contains a Lagrangian subspace.

Denote by $\varphi$ the map from $\mathscr V$ to
the dual space $\mathscr V^*$ given by $v\mapsto\langle v,\cdot\,\rangle$.
The space $\mathscr V/\ker\varphi$ is symplectic with respect to $\langle\,,\rangle$
(which is well-defined on the quotient space by construction),
and $\varphi$ transfers the symplectic structure to $\im\varphi\subset V^*$
via the natural isomorphism $\mathscr V/\ker\varphi\rightarrow\im\varphi$.

Let $\mathscr R\subset\mathscr V^*$ be a vector subspace. The following three conditions are equivalent:
\begin{enumerate}
\item the intersection $\mathscr R\cap\im\varphi$ is coisotropic;
\item the subspace $\mathrm{Ann}(\mathscr R)=
\{v\in\mathscr V\;;\;r(v)=0\ \forall r\in\mathscr R\}$ is isotropic;
\item the image of $\langle\,,\rangle\in\mathscr V^*\wedge\mathscr V^*$ in $(\mathscr V^*/\mathscr R)\wedge(\mathscr V^*/\mathscr R)$
under the natural projection $\mathscr V^*\wedge\mathscr V^*\rightarrow(\mathscr V^*/\mathscr R)\wedge(\mathscr V^*/\mathscr R)$
iz zero.
\end{enumerate}

\begin{defi}
A subspace $\mathscr R\subset\mathscr V^*$ will be called \emph{rich} with respect to
$\langle\,,\rangle$ if the three equivalent conditions above are satisfied.
\end{defi}

\begin{defi}
A restriction space for an interval exchange transformation
$T_{\pi,\mathbf a}$ is said to be \emph{rich} if it is rich with respect
to the bilinear form $\langle\,,\rangle_\pi$ on $\mathbb R^n$ defined by
\begin{equation}\label{bilinearform}\langle e_i,e_j\rangle_\pi=\left\{\begin{array}{rl}
1&\text{if }i<j\text{ and }\pi^{-1}(i)>\pi^{-1}(j),\\
-1&\text{if }i>j\text{ and }\pi^{-1}(i)<\pi^{-1}(j),\\
0&\text{otherwise},
\end{array}\right.
\end{equation}
where $e_1,\ldots,e_n$ is the standard basis of $\mathbb R^n$.
\end{defi}

Now we describe the topological meaning of the above definitions. Let $(\Sigma_{\pi,\mathbf a},\omega_{\pi,\mathbf a})$ be
the surface and the closed $1$-form associated with an interval exchange transformation $T_{\pi,\mathbf a}$
as described in Section~\ref{prelimsec}.

Let $\mathscr V$ be the first relative cohomology group $H^1(\Sigma_{\pi,\mathbf a},S_{\pi,\mathbf a};\mathbb R)$.
The dual vector space $\mathscr V^*$ is naturally identified
with $H_1(\Sigma_{\pi,\mathbf a},S_{\pi,\mathbf a};\mathbb R)$.
Let $e_1,\ldots,e_n$ be the basis in $\mathscr V$ dual to the basis $h_1,\ldots,h_n$
introduced in Section~\ref{restrictionsec}.

Then the bilinear form $\langle\,,\rangle_\pi$ defined by \eqref{bilinearform} is nothing else but
``the intersection form'' on $H^1(\Sigma_{\pi,\mathbf a},S_{\pi,\mathbf a};\mathbb R)$:
$$\langle\eta_1,\eta_2\rangle_\pi=\int\limits_{\Sigma_{\pi,\mathbf a}}\eta_1\wedge\eta_2$$
(provided that the orientation of $\Sigma_{\pi,\mathbf a}$ is chosen appropriately).

In the notation introduced in the beginning of the section the map $\varphi$ is the composition $p\circ\mathrm{PD}$,
where $\mathrm{PD}:H^1(\Sigma_{\pi,\mathbf a},S_{\pi,\mathbf a};\mathbb R)\rightarrow H_1(\Sigma_{\pi,\mathbf a}\setminus S_{\pi,\mathbf a};\mathbb R)$
is the Poincar\'e duality operator and $p:H_1(\Sigma_{\pi,\mathbf a}\setminus S_{\pi,\mathbf a};\mathbb R)
\rightarrow H_1(\Sigma_{\pi,\mathbf a};\mathbb R)\subset H_1(\Sigma_{\pi,\mathbf a},S_{\pi,\mathbf a};\mathbb R)$ the natural projection.

Thus, the symplectic spaces $\mathscr V/\ker\varphi$ and $\im\varphi$ are $H^1(\Sigma_{\pi,\mathbf a};\mathbb R)$
and $H_1(\Sigma_{\pi,\mathbf a};\mathbb R)$, respectively, and the respective symplectic forms on them
are the $\smile$ and $\frown$ products.

\begin{defi}
\emph{The Sah--Arnoux--Fathi (SAF) invariant of $T_{\pi,\mathbf a}$} is the following element of the rational vector space $\mathbb R\wedge_{\mathbb Q}\mathbb R$:
$$\saf(T_{\pi,\mathbf a})=\sum_{i=1}^na_i\wedge_{\mathbb Q}(\widetilde x_{\pi^{-1}(i)}-x_i)=\sum_{i<j}(a_i\wedge_{\mathbb Q}a_j-a_{\pi(i)}\wedge_{\mathbb Q}a_{\pi(j)}).$$
\end{defi}

The SAF invariant was introduced by P.\,Arnoux in \cite{Ar80} who also showed that $\saf(T_{\pi, \mathbf a})$ is an invariant of the measured foliation induced by~$\omega_{\pi,\mathbf a}$ on $\Sigma_{\pi,\mathbf a}$ (in particular, it is invariant under the Rauzy induction). 

The SAF invariant 
has many applications in the study of interval exchanges. In particular, it
has been used to characterize group properties of interval exchange transformations 
\cite{V} and to study the Veech group of translation surfaces~\cite{as09,ks}.

It is known that the SAF invariant vanishes for periodic (i.e.\ such that every orbit is finite) interval exchange transformations. Arnoux and Yoccoz constructed in \cite{AY} the first example of minimal and uniquely ergodic interval exchange transformation for which $\saf$ is equal to zero.

The Galois flux introduced in~\cite{McM} can be viewed as an instance of the $\saf$ invariant
in the case when the parameters of an interval exchange transformation belong to a quadratic field.

The following statement appears in~\cite{Ar80} with an attribution to G.\,Levitt.

\begin{prop}\label{saf=0}
The full restriction space $\mathscr R(T_{\pi,\mathbf a})$ is rich if and only if $\saf(T_{\pi,\mathbf a})=0$.
\end{prop}
\begin{proof}
Denote by $\mathscr R_0$ the intersection of $\mathscr R=\mathscr R(T_{\pi,\mathbf a})\subset
H_1(\Sigma_{\pi,\mathbf a},S_{\pi,\mathbf a};\mathbb R)$ with $H_1(\Sigma_{\pi,\mathbf a},S_{\pi,\mathbf a};\mathbb Q)$.
Since the form $\langle\,,\rangle_\pi$ has integer coefficients and $\mathscr R(T_{\pi,\mathbf a})$
is generated by integral relative cycles, $\mathscr R$ is a rich subspace of $H_1(\Sigma_{\pi,\mathbf a},S_{\pi,\mathbf a};\mathbb R)$
if and only if $\mathscr R_0$ is a rich subspace of $H_1(\Sigma_{\pi,\mathbf a},S_{\pi,\mathbf a};\mathbb Q)$.

By construction $\mathscr R_0$ is the kernel of the $\mathbb Q$-linear map $A:H_1(\Sigma_{\pi,\mathbf a},S_{\pi,\mathbf a};\mathbb Q)\rightarrow\mathbb R$
defined by $A(h_i)=a_i$, $i=1,\ldots,n$. The claim now follows from the fact that $A$ induces
a map $H_1(\Sigma_{\pi,\mathbf a},S_{\pi,\mathbf a};\mathbb Q)\wedge_{\mathbb Q}
H_1(\Sigma_{\pi,\mathbf a},S_{\pi,\mathbf a};\mathbb Q)\rightarrow\mathbb R\wedge_{\mathbb Q}\mathbb R$ that takes the form $\langle\,,\rangle_\pi$ to
$\saf(T_{\pi,\mathbf a})$.
\end{proof}

\begin{rema}
It immediately follows from Proposition~\ref{saf=0} that, for a fixed~$\pi$, the set of $\mathbf a$ such that
$\saf(T_{\pi,\mathbf a})=0$ has measure zero.
\end{rema}

\subsection{Asymptotic cycle and separating cycle}\label{asymptotic-n-separating}

Let $\pi\in S_n$ and $\mathbf a\in\Delta^{n-1}$ be fixed. Following A.\,Zorich~\cite{z97,z99}, for any $x\in[0,1)$, we denote by
$c_{\pi,\mathbf a,k}(x)$ the homology class from $H_1(\Sigma_{\pi,\mathbf a};\mathbb R)$ presented by the closed curve
$$\Bigl(\bigcup\limits_{i=1}^k\{T_{\pi,\mathbf a}^{i-1}(x)\}\times[0,1]\Bigr)\cup\alpha,\quad\text{where }\alpha=[x,T_{\pi,\mathbf a}^k(x)]\times\{0\}\text{ or }
[T_{\pi,\mathbf a}^k(x),x]\times\{0\},$$
under identifications~\eqref{identifications}. (The orientation is chosen so that ``vertical'' segments
are directed upwards.)

The following is a particular case of the notion of asymptotic cycle introduced by S.\,Schwartzman~\cite{sch}
in different terms.

\begin{defi}
The homology class $c_{\pi,\mathbf a}$ that is Poincar\'e dual to $[\omega_{\pi,\mathbf a}]\in H^1(\Sigma_{\pi,\mathbf a};\mathbb R)$
is called \emph{the asymptotic cycle of $T_{\pi,\mathbf a}$}.
\end{defi}

By definition the asymptotic cycle belongs to the absolute homology $H_1(\Sigma_{\pi,\mathbf a};\mathbb R)$,
which is a subspace of the relative homology $H_1(\Sigma_{\pi,\mathbf a},S_{\pi,\mathbf a};\mathbb R)$.

\begin{prop}\label{asymptoticcycle}
If $T_{\pi,\mathbf a}$ is uniquely ergodic, then for any $x\in[0,1)$ there exists a limit

\begin{equation}\label{limit}
\lim_{k\rightarrow\infty}\frac{c_{\pi,\mathbf a,k}(x)}k.\end{equation}
This limit does not depend on $x$ and is equal up to a non-zero
multiple
to the asymptotic cycle $c_{\pi,\mathbf a}$.
Moreover, the convergence is uniform with
respect to $x$.
\end{prop}

The existence of the limit~\eqref{limit} for almost all $x$ is a simple consequence
of Birkhoff's ergodic theorem. The fact that for a generic irreducible interval exchange transformation
the convergence is uniform and holds for all $x\in[0,1)$ was observed by Zorich~\cite{z97,z99}.
A proof of Proposition~\ref{asymptoticcycle} in exactly this form can be found in~\cite[Section~3]{Vi}.
Due to  the normalization condition $\sum_ia_i=1$ `a non-zero multiple' in the formulation
of Proposition~\ref{asymptoticcycle} is actually $\pm1$.

\begin{prop}\label{obvious}
The full restriction space of $T_{\pi,\mathbf a}$ is rich if and only if it contains the asymptotic cycle~$c_{\pi,\mathbf a}$.
\end{prop}

\begin{proof}
The space $\mathscr W=\mathscr R(T_{\pi,\mathbf a})\cap H_1(\Sigma_{\pi,\mathbf a};\mathbb R)$ is generated by all $c\in
H_1(\Sigma_{\pi,\mathbf a};\mathbb Z)$ such that $c\frown c_{\pi,\mathbf a}=0$. If $\mathscr W$
is coisotropic this implies $c_{\pi,\mathbf a}\in\mathscr W$.

If $\mathscr W$ is not coisotropic, then there exists $s\in H_1(\Sigma_{\pi,\mathbf a};\mathbb Z)\setminus\mathscr W$ such that
$s\frown r=0$ for all $r\in\mathscr W$.
Since $s\notin\mathscr W$ we have $s\frown c_{\pi,\mathbf a}\ne0$.
Therefore, $c_{\pi,\mathbf a}\notin\mathscr W$, which implies $c_{\pi,\mathbf a}\notin\mathscr R(T_{\pi,\mathbf a})$.
\end{proof}

\begin{defi}\label{separating-def}
An element $s\in H_1(\Sigma_{\pi,\mathbf a};\mathbb Z)$ such that $s\frown r=0$ for
all $r\in\mathscr R(T_{\pi,\mathbf a})\cap H_1(\Sigma_{\pi,\mathbf a};\mathbb R)$ and $s\frown c_{\pi,\mathbf a}\ne0$ will be called \emph{a separating cycle}
for $T_{\pi,\mathbf a}$. As we have just seen, it exists if and only if the full restriction space $\mathscr R(T_{\pi,\mathbf a})$ is poor.
\end{defi}

\begin{prop}\label{prop-separ}
If a separating cycle for $T_{\pi,\mathbf a}$ can be presented by a closed transversal of the foliation~$\mathcal F_{\pi,\mathbf a}$ such
that it intersects all separatrices (and hence, all leaves) of $\mathcal F_{\pi,\mathbf a}$, then $T_{\pi,\mathbf a}$ is minimal.
\end{prop}

\begin{proof}
Let $\xi$ be a transversal representing a separating cycle $s$ such that $\xi$ intersects all saddle connections of~$\mathcal F_{\pi,\mathbf a}$.
If $\mathcal F_{\pi,\mathbf a}$ is not minimal, then there must be a realization $\rho$ of a loop restriction $r$.
It must intersect~$\xi$, and the contributions of all intersections to $s\frown r$
will have the same sign. So, we will have $s\frown r\ne0$, a contradiction.
\end{proof}

\begin{rema}
If a homology class $s\in H_1(\Sigma_{\pi,\mathbf a};\mathbb Z)$ can be presented by
a transversal intersecting all the leaves of~$\mathcal F_{\pi,\mathbf a}$, then
we always have $s\frown c_{\pi,\mathbf a}\ne0$.
\end{rema}

\begin{rema} A separating
cycle can alternatively be defined as an element $s$ of $H_1(\Sigma_{\pi,\mathbf a}\setminus S_{\pi,\mathbf a};\mathbb Z)$
such that $s\frown r=0$ for any restriction $r$ and $s\frown c_{\pi,\mathbf a}\ne0$.

In this setting, the assertion of Proposition~\ref{prop-separ} will be that $T_{\pi,\mathbf a}$ satisfies
Keane's condition.
\end{rema}

\begin{exam}\label{exam-seapr}
In Example~\ref{exam-minimal}, the set $S_{\pi,\mathbf a}$ consists of a single point, in which $\mathcal F_{\pi,\mathbf a}$ has a double saddle singularity.
There is a single restriction, up to a multiple, which is $r=3h_1-h_2-h_4$. It can also be presented
by the closed curve $\rho$ consisting of the following three oriented straight line segments (under identifications~\eqref{identifications}):
$$\rho_1=[(a_2,0),(a_1,1)],\ \rho_2=[(a_2+a_4-a_1,0),(a_1+a_2,1)],\ \rho_3=[(a_2+a_3+a_4,0),(a_2+a_3+a_4,1)].$$

To see that $\rho$ represents $3h_1-h_2-h_4$ we compute
$$\int_\rho\,dx=\int_{\rho_1}\,dx+\int_{\rho_2}\,dx+\int_{\rho_3}\,dx=
(a_1-a_2)+(2a_1-a_4)+0=3a_1-a_2-a_4,$$
which holds for any $\mathbf a\in\Delta^3$.

The cycle $s=2h_1-h_2+h_3$ is separating, and if $a_2<a_1$
it can be presented by a transversal $\xi$ obtained by a small deformation from the
closed curve consisting of the following oriented straight line segments:
$$\xi_1=[(a_2,0),(a_1,1)],\ \xi_2=[(a_2,0),(a_1+a_2,1)],\ \xi_3=[(a_1+a_2,1),(a_1+a_2+a_3,1)],$$
see Fig.~\ref{sepcycle}. 
Indeed, for any $\mathbf a\in\Delta^3$ we have
$$\int_\xi dx=\int_{\xi_1}dx+\int_{\xi_2}dx+\int_{\xi_3}dx=
(a_1-a_2)+a_1+a_3=2a_1-a_2+a_3,$$
which implies that $\xi$ represents $s$.
The transversal $\xi$ is disjoint from $\rho$, so, we have $s\frown r=0$.

It is quite obvious (consult Fig.~\ref{sepcycle}) that separatrices emanating
from $(x_1,1)$, $(x_3,1)$ in the downward direction and from $(\widetilde x_1,0)$, $(\widetilde x_2,0)$
in the upward direction hit the transversal $\xi$. One can see that the two remaining separatrices can avoid
meeting $\xi$ only if they form a saddle connection. This means that for some $k\geqslant1$
we have $T_{\pi,\mathbf a}^{-k}(x_2)=x_2+k(a_1+a_3)=\widetilde x_3$,
and $a_1+k(a_1+a_3)=a_4+a_3$. This is an integral restriction that $T_{\pi,\mathbf a}$ is supposed
not to satisfy.

Thus $\xi$ intersects all the separatrices and hence all the leaves of $\mathcal F_{\pi,\mathbf a}$.
Therefore, according
to Proposition~\ref{prop-separ} the foliation $\mathcal F_{\pi,\mathbf a}$ is minimal.
\begin{figure}[ht]
\centerline{\includegraphics{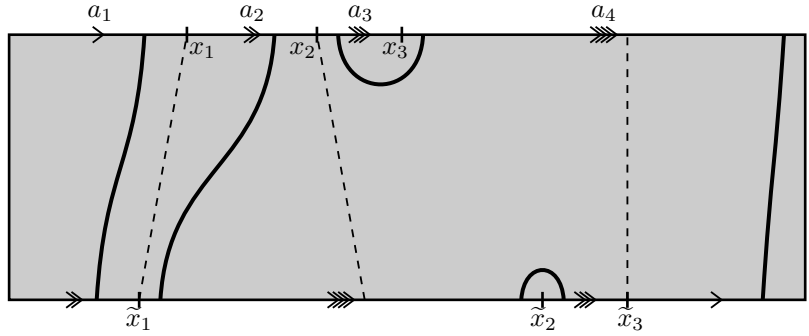}\put(-281.5,117){$a_1$}\put(-224,117){$a_2$}\put(-183,117){$a_3$}\put(-91,117){$a_4$}%
\put(-243,103){$x_1$}\put(-205,103){$x_2$}\put(-170,103){$x_3$}\put(-267,0){$\widetilde x_1$}\put(-114,0){$\widetilde x_2$}%
\put(-81,0){$\widetilde x_3$}}
\caption{A transversal representing a separating cycle (bold line) and the restriction cycle (dashed line)
in Example~\ref{exam-seapr}}\label{sepcycle}
\end{figure}

It also follows from Theorem~\ref{uniquelyergodictheo} below that the transformation $T_{\pi,\mathbf a}$ is
uniquely ergodic for almost all $\mathbf a$ such that $a_1>a_2$ and $\rho(\mathbf a)=0$.

If $a_1<a_2$, then this construction fails because $\int_{\xi_1}<0$ and $\int_{\xi_2}>0$, hence $\xi_1\cup\xi_2\cup\xi_3$
cannot be made transverse to the foliation by a small deformation.
\end{exam}

\section{Stability of minimal interval exchange transformations with poor restrictions}
\label{poor->stable}

\begin{defi}
Let $\mathscr R$ be a restriction space of a minimal interval exchange transformation $T_{\pi,\mathbf a}$. We say that
$T_{\pi,\mathbf a}$ is \emph{$\mathscr R$-stably minimal} if there exists an open neighborhood $U$ of $\mathbf a$ in
$\Delta^{n-1}$ such that $T_{\pi,\mathbf a'}$ is minimal whenever $\mathbf a'\in U$ and $\mathscr R(T_{\pi,\mathbf a'})=\mathscr R$.
\end{defi}

\begin{theo}\label{stabilityofminimal}
Let $T_{\pi,\mathbf a}$ be a minimal and uniquely ergodic interval exchange transformation such
that the full restriction space $\mathscr R(T_{\pi,\mathbf a})$ is poor. 
Let $\mathscr R_0$ be a subspace of $\mathscr R(T_{\pi,\mathbf a})$ generated
by all restrictions realized by $T_{\pi,\mathbf a}$ (see Definition~\ref{realized}).
Then the transformation $T_{\pi,\mathbf a}$
is $\mathscr R$-stably minimal for
any restriction space $\mathscr R$ such that $\mathscr R_0\subset\mathscr R\subset\mathscr R(T_{\pi,\mathbf a})$.
\end{theo}

\begin{proof}
If $\mathscr R=0$, then the assertion follows from Keane's result~\cite{k75}. In what follows we assume
$0\ne\mathscr R\subset\mathscr R(T_{\pi,\mathbf a})$, and $\mathscr R=\mathscr R(T_{\pi,\mathbf a})\ne0$ if $T_{\pi,\mathbf a}$ does
not satisfy Keane's condition.

By Imanishi's theorem~\cite{i79} the foliation $\mathcal F_{\pi,\mathbf a}$, and hence the transformation $T_{\pi,\mathbf a}$
is not minimal if and only if some union of realizations of restrictions that is contained in finitely many leaves of $\mathcal F_{\pi,\mathbf a}$
cut the surface $M_{\pi,\mathbf a}$~into two non-trivial pieces. Keane's condition is equivalent to the absence of
any realized restrictions.

Thus, in order to prove that $T_{\pi,\mathbf a'}$ is minimal it suffices to show that $T_{\pi,\mathbf a'}$ does not
realize any simple restriction that is not already realized by $T_{\pi,\mathbf a}$.

Let $s\in H_1(\Sigma_{\pi,\mathbf a};\mathbb Z)$ be a separating cycle. Without loss of generality we can assume $s\frown c_{\pi,\mathbf a}>0$.
Take an arbitrary preimage $\widetilde s\in H_1(\Sigma_{\pi,\mathbf a}\setminus S_{\pi,\mathbf a};\mathbb Z)$ under
the natural projection $H_1(\Sigma_{\pi,\mathbf a}\setminus S_{\pi,\mathbf a};\mathbb Z)\rightarrow H_1(\Sigma_{\pi,\mathbf a};\mathbb Z)$.
Note that the $\frown$ product is a well defined pairing $H_1(\Sigma_{\pi,\mathbf a}\setminus S_{\pi,\mathbf a};\mathbb Z)\times
H_1(\Sigma_{\pi,\mathbf a},S_{\pi,\mathbf a};\mathbb Z)\rightarrow\mathbb Z$.

\begin{lemm}\label{arc-restriction-lem}
There exists a constant $C$ such that $|\widetilde s\frown r|<C$ for any simple restriction $r\in\mathscr R(T_{\pi,\mathbf a})$.
\end{lemm}

\begin{proof}
If $r$ is an absolute cycle, then $s\frown r=0$ by Definition~\ref{separating-def}.

If $r_1$ and $r_2$ are two arc restrictions such that $\partial r_1=\partial r_2$, then $r_1-r_2\in H_1(\Sigma_{\pi,\mathbf a};\mathbb Z)$
is an absolute cycle, so, we have $(r_1-r_2)\frown\widetilde s=(r_1-r_2)\frown s=0$. Therefore,
$r_1\frown\widetilde s=r_2\frown\widetilde s$. The assertion of the lemma now follows from the fact
that there are only finitely many different possibilities for $\partial r$ if $r$ is an arc restriction.
\end{proof}

For $x,y\in S_{\pi,\mathbf a}$ denote by $r_{xy}$ the relative $1$-cycle presented by the
arc $[x,y]\times\{1\}$ if $x<y$ and $[y,x]\times\{1\}$ otherwise directed from $x$ to $y$.
If $r$ is an arc restriction with $\partial r=y-x$ we denote by $\widehat r$
the absolute cycle $r-r_{xy}\in H_1(\Sigma_{\pi,\mathbf a};\mathbb Z)$.
If $r$ is a loop restriction we set $\widehat r=r$.

It follows from Proposition~\ref{asymptoticcycle} and Lemma~\ref{arc-restriction-lem} that, for large enough $k$,
we have
\begin{equation}\label{sfrownr}
s\frown c_{\pi,\mathbf a,j}(x)>s\frown \widehat r\end{equation}
whenever $j\in[k,2k]\cap\mathbb Z$, $x\in[0,1)$, and $r$ is a simple restriction. We fix such a large $k$ from now on.

Denote $\mathcal C_{\pi,\mathbf a,j}=\{c_{\pi,\mathbf a,j}(x)\;;\;x\in[0,1)\}$.

\begin{lemm}\label{definitionofU}
There exists a neighborhood $U$ of $\mathbf a$ in $\Delta^{n-1}$ such that 
whenever $\mathbf a'\in U$ and $\mathscr R(T_{\pi,\mathbf a'})=\mathscr R$ we have:
\begin{enumerate}
\item
$\mathcal C_{\pi,\mathbf a',j}=\mathcal C_{\pi,\mathbf a,j}$ for all $j\in[1,2k]\cap\mathbb Z$;
\item
$T_{\pi,\mathbf a'}$ does not realize any restriction of length $\leqslant 2k$ that
is not realized by $T_{\pi,\mathbf a}$.
\end{enumerate}
\end{lemm}

\begin{proof}
One can see that, for any $j$, the finite set $\mathcal C_{\pi,\mathbf a,j}$ is defined by the permutation $\pi$
and the relative positions of points from $\bigcup_{i=0}^{j-1}(T_{\pi,\mathbf a})^i(S_{\pi,\mathbf a})$
(we abuse notation here slightly by thinking of $S_{\pi,\mathbf a}$ as a subset of $[0,1)$).
If $\mathbf a$ is perturbed slightly, then the order of points in this union does not change provided
that all coincidences, if any,  between these points are preserved.

Coincidences may occur only if Keane's condition is not satisfied.
Indeed, an equation $(T_{\pi,\mathbf a})^i(x)=y$ with $x,y\in S_{\pi,\mathbf a}$ means that an arc
restriction is realized in the singular leaf passing through $x$.

If $T_{\pi,\mathbf a'}$ satisfies
all restrictions realized by $T_{\pi,\mathbf a}$ and
$\mathbf a'$ is close enough to $\mathbf a$, then $(T_{\pi,\mathbf a'})^i(x)=y$ holds true, too.

If $\mathbf a'$ is sufficiently close to $\mathbf a$ then no new coincidences occur in
$\bigcup_{i=0}^{2k}(T_{\pi,\mathbf a'})^i(S_{\pi,\mathbf a'})$ compared to $\bigcup_{i=0}^{2k}(T_{\pi,\mathbf a})^i(S_{\pi,\mathbf a})$,
which implies Condition~(2) from the Lemma assertion.
\end{proof}

We fix $U$ as in Lemma~\ref{definitionofU}. Let $\mathbf a'\in U$, $\mathscr R(T_{\pi,\mathbf a'})=\mathscr R$.

\begin{lemm}\label{j>=k}
For any simple restriction $r$ and any $j\geqslant k$ we have
\begin{equation}
s\frown c_{\pi,\mathbf a',j}(x)>s\frown \widehat r.
\end{equation}
\end{lemm}
\begin{proof}
We apply induction in $[j/k]$, where $[x]$ stands for the integral part of $x$.
For $k\leqslant j\leqslant2k$ the assertion follows from Lemma~\ref{definitionofU} and inequality~\eqref{sfrownr}.

The induction step is obtained
by applying the following relation
$$c_{\pi,\mathbf a',j+k}(x)=c_{\pi,\mathbf a',j}(x)+c_{\pi,\mathbf a',k}\bigl((T_{\pi,\mathbf a'})^j(x)\bigr)\in\mathcal C_{\pi,\mathbf a',j}+\mathcal C_{\pi,\mathbf a',k}.$$
The inequality
$$s\frown c>\max_{\parbox{1.4 cm}{\tiny $r$ is a simple restriction}}s\frown\widehat r$$
holds for any $c\in C_{\pi,\mathbf a',j}$ by the induction hypothesis and for any $c\in C_{\pi,\mathbf a',k}$
by the induction base. Therefore, it also holds
for any $c\in\mathcal C_{\pi,\mathbf a',j+k}\subset\mathcal C_{\pi,\mathbf a',j}+\mathcal C_{\pi,\mathbf a',k}$.
\end{proof}

We are now ready to conclude the proof of Theorem~\ref{stabilityofminimal}.

Suppose that $T_{\pi,\mathbf a'}$ realizes some simple restriction that is not realized by $T_{\pi,\mathbf a}$.
Let $r$ be such a restriction with minimal possible length, and $j$ the length of $r$.
By the choice of $U$ we must have $j>2k$.

There must exist $x,y\in S_{\pi,\mathbf a'}$ (possibly $x=y$) such that $(T_{\pi,\mathbf a'})^j(x)=y$ holds,
and $r$ is presented by an arc (or loop if $x=y$) contained in the singular fiber through $x$.

We have $c_{\pi,\mathbf a',j}(x)=\widehat r$, which contradicts Lemma~\ref{j>=k}.

Therefore, $T_{\pi,\mathbf a'}$ does not realize any simple restriction that is not realized by $T_{\pi,\mathbf a}$ and, hence,
is minimal.
\end{proof}

Barak Weiss drew our attention to the fact that the results of paper~\cite{mw} allow us to prove
an analogue of the Masur--Veech theorem \cite{m82,v82} (Keane's conjecture~\cite{k77})
in our settings. The next theorem is essentially due to him though we modified
the formulation and the proof from his original suggestion in order to keep
the style of the paper uniform.

\begin{theo}\label{uniquelyergodictheo}
Let $T_{\pi,\mathbf a}$ be a minimal and uniquely ergodic interval exchange transformation such
that the full restriction space $\mathscr R(T_{\pi,\mathbf a})$ is poor. 
Let $\mathscr R_0$ be a subspace of $\mathscr R(T_{\pi,\mathbf a})$ generated
by all restrictions realized by $T_{\pi,\mathbf a}$ (see Definition~\ref{realized}),
and $\mathscr R$ is a restriction space for $T_{\pi,\mathbf a}$ such that $\mathscr R_0\subset\mathscr R\subset\mathscr R(T_{\pi,\mathbf a})$.
Then there exists an open neighborhood $V$ of $\mathbf a$ in $\{\mathbf b\in\Delta^{n-1}\;;\;\mathscr R(T_{\pi,\mathbf b})\supset\mathscr R\}$
such that $T_{\pi,\mathbf a'}$ is uniquely ergodic for almost all $\mathbf a'\in V$
with respect to the Lebesgue measure.
\end{theo}

\begin{proof}
First we consider the case when $T_{\pi,\mathbf a}$ satisfies Keane's condition, i.e.\ when $\mathcal F_{\pi,\mathbf a}$ has
no saddle connections. Denote $P=\{\mathbf b\in\Delta^{n-1}\;;\;\mathscr R(T_{\pi,\mathbf b})\supset\mathscr R\}$.
$P$ is a polyhedron in an affine subspace of $\mathbb R^n$ with the latter being
identified with $H^1(\Sigma_{\pi,\mathbf a},S_{\pi,\mathbf a};\mathbb R)$ (see Section~\ref{rich-and-poor}),
and $\mathbf a$ is an interior point of $P$.

For a neighborhood $U\subset P$ of $\mathbf a$ we can identify all surfaces $\Sigma_{\pi,\mathbf a'}$, $\mathbf a'\in U$,
with a fixed surface $\Sigma$ so that $S_{\pi,\mathbf a'}$ is identified with a fixed subset $S\subset\Sigma$,
and the form $\omega_{\pi,\mathbf a'}$ viewed as a $1$-form on $\Sigma$
depends on $\mathbf a'$ continuously (in $C^1$ topology, say). We may also assume
that $\omega_{\pi,\mathbf a'}$ does not depend on $\mathbf a'$ in a small neighborhood of
the subset $S$, and the image of the oriented interval $[0,1]\times\{1/2\}$ under identifications~\eqref{identifications} in $S_{\pi,\mathbf a'}$
is identified with a fixed oriented curve $\gamma$ in $S$ for all $\mathbf a'\in U$. It defines a relative homology class
$[\gamma]\in H_1(\Sigma,S;\mathbb Z)$.

Let $s\in H_1(\Sigma_{\pi,\mathbf a};\mathbb Z)$ be a separating cycle for $T_{\pi,\mathbf a}$. It has a preimage $s'$
in $H_1(\Sigma\setminus S;\mathbb Z)$ such that $s'\frown r=0$ for all restrictions $r\in\mathscr R$.
Denote by $c'_{\pi,\mathbf a}$ the homology class from $H_1(\Sigma\setminus S;\mathbb R)$
that is Poincar\'e dual to the relative cohomology class $[\omega_{\pi,\mathbf a}]\in H^1(\Sigma,S;\mathbb R)$.
(The image of~$c'_{\pi,\mathbf a}$ in $H_1(\Sigma;\mathbb R)$ under the projection induced by
the natural inclusion $\Sigma\setminus S\hookrightarrow\Sigma$ is the asymptotic cycle $c_{\pi,\mathbf a}$.)

Let $s''=s'-([\gamma]\frown s')c'_{\pi,\mathbf a}$. We will have $[\gamma]\frown s''=0$, and still $s''\frown r=0$ for any
restriction $r\in\mathscr R$, and $s''\frown c_{\pi,\mathbf a}\ne0$. Denote by $\mathbf b\in H^1(\Sigma,S;\mathbb R)$
the cohomology class that is Poincar\'e dual to $s''$. It can be viewed as a tangent vector to $P$ (at $\mathbf a$ or
any other point as $P$ is a polyhedron in an affine space)
since $\mathbf b([\gamma])=0$ and $\mathbf b(r)=0$ for any $r\in\mathscr R$. We also have
$\mathbf b(c_{\pi,\mathbf a})\ne0$. By changing the sign of $\mathbf b$ we may assume
$\mathbf b(c_{\pi,\mathbf a})>0$.

It follows from \cite[Theorem~1.1]{mw} that the cohomology class $\mathbf b$ can be
represented by a closed $1$-form~$\sigma$ on $\Sigma$
such that $\omega_{\pi,\mathbf a}+\mathbbm i\sigma$ is a holomorphic $1$-form
with respect to some Riemann surface structure on $\Sigma$ (this complex structure
need not agree with the smooth structure on $\Sigma$ at $S$). Since $\gamma$
is transverse to the foliation defined by $\omega_{\pi,\mathbf a}$, the $1$-form $\sigma$
can be chosen so that $\sigma|_\gamma$ vanishes at the endpoints of $\gamma$.

The $1$-form $\omega_{\pi,\mathbf a}+\mathbbm i\sigma$ defines a locally Euclidean metric
$ds^2=|\omega_{\pi,\mathbf a}+\mathbbm i\sigma|^2$ on $\Sigma\setminus S$. With respect to this metric, the transversal $\gamma$
intersects the leaves of $\mathcal F_{\pi,\mathbf a}$ at an angle bounded from
below by a positive constant $\delta$. Since $\omega_{\pi,\mathbf a'}$ depends on $\mathbf a'$
continuously and coincides with $\omega_{\pi,\mathbf a}$ in a small neighborhood of $S$ when $\mathbf a'$
is close enough to $\mathbf a$, there is an open neighborhood $V\subset U$ of $\mathbf a$ such
that for all $\mathbf a'\in V$ the $1$-form $\omega_{\pi,\mathbf a'}+\mathbbm i\sigma$ is
a holomorphic $1$-form for some complex structure on $S$ (depending on $\mathbf a'$), and, with respect to
the locally Euclidean metric $|\omega_{\pi,\mathbf a'}+\mathbbm i\sigma|^2$, the curve $\gamma$ 
remains transverse to
the leaves of~$\mathcal F_{\pi,\mathbf a'}$ and intersects them at an angle bounded from below by some $\delta>0$.
This $\delta$ can
be chosen to be independent of $\mathbf a'$, though this is actually not important.

We have that $\gamma$ is transverse to the leaves of the foliation defined by the
$1$-form
$$\zeta_{\mathbf a',\phi}=\re\big(e^{\mathbbm i\phi}(\omega_{\pi,\mathbf a'}+\mathbbm i\sigma)\big)$$
if $\mathbf a'\in V$ and $|\phi|<\delta$. If, additionally, $\mathbf a'+\tan\phi\cdot\mathbf b$
lies in the interior of $P$, then the first return map of the foliation defined by $\zeta_{\mathbf a',\phi}$ on $\gamma$
is, after an appropriate choice of the coordinate on $\gamma$, an interval exchange transformation $T_{\pi,\mathbf a'+\tan\phi\cdot\mathbf b}$.

By Kerckhoff--Masur--Smillie~\cite{kms}, for a fixed $\mathbf a'\in V$, the foliation defined by $\zeta_{\mathbf a',\phi}$
is uniquely ergodic for almost all $\phi$. It means that if $\ell$ is a straight line in $\mathbf R^n$ having direction $\mathbf b$ and passing
through a point in $V$, then for almost any $\mathbf a'\in\ell\cap V$ the interval exchange transformation $T_{\pi,\mathbf a'}$
is ergodic. Fubini's theorem implies that $T_{\pi,\mathbf a'}$ is uniquely ergodic for almost all $\mathbf a'\in V$
provided that the set of all such $\mathbf a'$ is measureable.

By using the Rauzy induction~\cite{rauzy79} one can show that the set of all $\mathbf a\in\Delta^{n-1}$ such that $T_{\pi,\mathbf a}$ is uniquely ergodic
is a Borel set. Therefore, so is the intersection of this set with $P$.

We have completed the proof under the Keane's condition assumption.

Now we reduce the general case to the one when Keane's condition is satisfied.
We keep using notation $P=\{\mathbf b\in\Delta^{n-1}\;;\;\mathscr R(T_{\pi,\mathbf b})\supset\mathscr R\}$.
There is a small neighborhood $U\subset P$ of $\mathbf a$ such that whenever $\mathbf b\in U$
the interval exchange transformation $T_{\pi,\mathbf b}$ realizes any restriction realized by $T_{\pi,\mathbf a}$.

For $\mathbf b\in U$, denote by $\widetilde\Sigma_{\pi,\mathbf b}$ the singular surface obtained from $\Sigma_{\pi,\mathbf b}$
by collapsing to a point each saddle connection realizing a restriction from $\mathscr R_0$.
In general,
$\widetilde\Sigma_{\pi,\mathbf b}$ may have singularities. Namely, if~$T_{\pi,\mathbf a}$ realizes some loop restrictions,
then $\widetilde\Sigma_{\pi,\mathbf b}$, $\mathbf b\in U$, will have finitely many points with a neighborhood
that has the form of a wedge sum of several open discs. Thus, $\widetilde\Sigma_{\pi,\mathbf b}$ can be obtained
from a surface, which we denote by $\widehat\Sigma_{\pi,\mathbf b}$, by making finitely many identifications
of points.

Thus, we have two projections:
$$\Sigma_{\pi,\mathbf b}\stackrel p\longrightarrow\widetilde\Sigma_{\pi,\mathbf b}\stackrel{\widehat p}\longleftarrow\widehat\Sigma_{\pi,\mathbf b},$$
of two smooth surfaces to a singular one. The first projection, $p$, collapses all the saddle connections, and the second, $\widehat p$,
is one-to-one outside of a finite subset.

We denote by $\widetilde S_{\pi,\mathbf b}\subset\widetilde\Sigma_{\pi,\mathbf b}$ the image of
$S_{\pi,\mathbf b}$ under the projection $p$,
and by $\widehat S_{\pi,\mathbf b}\subset\widehat\Sigma_{\pi,\mathbf b}$ the preimage of
$\widetilde S_{\pi,\mathbf b}$ under $\widehat p$.
The kernel of the map $p_*:H_1(\Sigma_{\pi,\mathbf b},S_{\pi,\mathbf b};\mathbb R)\rightarrow
H_1(\widetilde\Sigma_{\pi,\mathbf b},\widetilde S_{\pi,\mathbf b};\mathbb R)$ induced by the projection~$p$ is clearly $\mathscr R_0$,
and the projection $\widehat p$
induces an isomorphism $\widehat p_*:H_1(\widehat \Sigma_{\pi,\mathbf b},\widehat S_{\pi,\mathbf b};\mathbb R)\rightarrow
H_1(\widetilde\Sigma_{\pi,\mathbf b},\widetilde S_{\pi,\mathbf b};\mathbb R)$. Thus we have
a natural epimorphism $\widehat p_*^{-1}\circ p_*:H_1(\Sigma_{\pi,\mathbf b},S_{\pi,\mathbf b};\mathbb R)\rightarrow
H_1(\widehat \Sigma_{\pi,\mathbf b},\widehat S_{\pi,\mathbf b};\mathbb R)$, which we denote by $\iota$,
whose kernel is $\mathscr R_0$.

There is an obvious way to transfer the $1$-form $\omega_{\pi,\mathbf b}$ from $\Sigma_{\pi,\mathbf b}$
to $\widehat\Sigma_{\pi,\mathbf b}$. The obtained $1$-form on $\widehat\Sigma_{\pi,\mathbf b}$ will
be denoted by $\widehat\omega_{\pi,\mathbf b}$, and the foliation it defines by $\widehat{\mathcal F}_{\pi,\mathbf b}$.
Clearly, $\widehat{\mathcal F}_{\pi,\mathbf b}$ is minimal if and only if so is $\mathcal F_{\pi,\mathbf b}$.

By construction the foliation $\widehat{\mathcal F}_{\pi,\mathbf a}$ has no saddle connections, so
it induces an interval exchange map that satisfies Keane's condition
at every transversal connecting two points from $\widehat S_{\pi,\mathbf a}$.
We choose such a transversal so that, additionally, it intersects all the leaves of $\widehat{\mathcal F}_{\pi,\mathbf a}$.
We can use this transversal for all $\mathbf b$ close enough to $\mathbf a$.
Let $T_{\widehat\pi,\widehat{\mathbf b}}$ be interval exchange map induced at this transversal
by $\widehat{\mathcal F}_{\pi,\mathbf b}$. We
can think of $\widehat\Sigma_{\pi,\mathbf b}$ being identified with $\Sigma_{\widehat\pi,\widehat{\mathbf b}}$.

For a cycle $c\in H_1(\Sigma_{\pi,\mathbf b},S_{\pi,\mathbf b};\mathbb Z)$, we obviously
have $\int_c\omega_{\pi,\mathbf b}=\int_{\iota(c)}\widehat\omega_{\pi,\mathbf b}$.
Therefore, the restriction space $\mathscr R(T_{\widehat\pi,\widehat{\mathbf b}})$
is naturally identified with $\mathscr R(T_{\pi,\mathbf b})/\mathscr R_0$.

The image $\iota(c)$ of an absolute cycle $c\in H_1(\Sigma_{\pi,\mathbf b};\mathbb R)$ is also
an absolute cycle, i.e. lies in $H_1(\widehat\Sigma_{\pi,\mathbf b};\mathbb R)$, if and only
if 
\begin{equation}\label{absolute}
c\frown r=0\text{ for any }r\in\mathscr R_0\cap H_1(\Sigma_{\pi,\mathbf b};\mathbb R).\end{equation}
This holds, in particular, for the asymptotic cycle $c_{\pi,\mathbf b}$ of $T_{\pi,\mathbf b}$,
and one can see that $\iota(c_{\pi,\mathbf b})$ is the asymptotic cycle of $T_{\widehat\pi,\widehat{\mathbf b}}$.
Condition~\eqref{absolute} also holds for any separating cycle $s$ of $T_{\pi,\mathbf b}$,
and the image $\iota(s)$ is then a separating cycle for $T_{\widehat\pi,\widehat{\mathbf b}}$.

Therefore, if $\mathbf b=\mathbf a$, then the full restriction space $\mathscr R(T_{\widehat\pi,\widehat{\mathbf b}})=\iota(\mathscr R)$
is poor. Varying $\mathbf b$ in a small neighborhood of $\mathbf a$ so that $\mathscr R(T_{\pi,\mathbf b})\supset\mathscr R_0$ holds
is equivalent to varying $\widehat{\mathbf b}$ in a small neighborhood of $\widehat{\mathbf a}$.
Thus, the passage from $T_{\pi,\mathbf a}$ to $T_{\widehat\pi,\widehat{\mathbf a}}$ described above
reduces the general case to the case when Keane's condition is satisfied.
\end{proof}

The assumption on the unique ergodicity of $T_{\pi,\mathbf a}$ in Theorems~\ref{stabilityofminimal} and \ref{uniquelyergodictheo}
can be weakened by requiring only that there exists a cycle $s\in H_1(\Sigma_{\pi,\mathbf a}\setminus S_{\pi,\mathbf a};\mathbb Z)$
such that 
\begin{enumerate}
\item
$s\frown r=0$ for all $r\in\mathscr R(T_{\pi,\mathbf a})$ and
\item
$s\frown c>0$ for any cycle $c$ that can be obtained as the limit in~\eqref{limit} for some $x$.
\end{enumerate}
This is a `weak' version of a separating cycle.
Due to Minsky--Weiss~\cite{mw}, condition~(2) holding either for $s$ or $-s$ can be shown
to be equivalent to the existence of a closed transversal representing $s$
and intersecting all regular leaves of $\mathcal F_{\pi,\mathbf a}$.

The proof of Theorem~\ref{uniquelyergodictheo} generalizes quite directly with the weaker assumption.
Our proof of Theorem~\ref{stabilityofminimal} does not generalize directly,
but there is another proof that does. Namely, instead of applying Proposition~\ref{asymptoticcycle}
one can show that the existence of a separating cycle (in the weak sense) is a stable property and that
it implies minimality (cf.~Proposition~\ref{prop-separ}).

The unique ergodicity requirement as well as the just mentioned weaker assumption
look artificial to us in the present context. We find it plausible that
such assumptions can be dropped completely in both theorems. So, we pose

\begin{conj}\label{conj1}
Theorems~\ref{stabilityofminimal} and \ref{uniquelyergodictheo} remain true without the assumption on the
unique ergodicity of $T_{\pi,\mathbf a}$.
\end{conj}

We conclude the section by a remark about dimensions.
As follows from the definition, a restriction space that is a subspace of a poor restriction space is also poor,
however, richness or poorness of a restriction space is not strongly related to its dimension. The following statement is easy.

\begin{prop}\label{maximalpoor}
(i) Any minimal rich restriction space has dimension equal to the genus of the suspension surface.

(ii) Any maximal poor restriction space has codimension two.
\end{prop}

Taking into account that the genus of the suspension surface is bounded from above by $n/2$ we see that
maximal poor restriction spaces are larger in dimension than minimal rich ones if $n>4$.

\begin{exam}
Consider in more detail Example~\ref{examplenovak} from the Introduction. The restriction space $\mathscr R$
is $(n-2)$-dimensional and generated by $h_1-h_3,h_2-h_4,\ldots,h_{n-2}-h_n$ in the notation of Section~\ref{restrictionsec}.

The restriction space
$$\langle h_1-h_3,h_2-h_4,\ldots,h_{n-2}-h_n\rangle$$
for interval exchange transformations with permutation~\eqref{permnov} is poor.

Indeed, the space $\mathrm{Ann}(\mathscr R)$ is $2$-dimensional and generated by
$e_1+e_3+\ldots+e_{n-1}$ and $e_2+e_4+\ldots+e_n$. 
For the bilinear form~\eqref{bilinearform} we have
$$\langle e_1+e_3+\ldots+e_{n-1},e_2+e_4+\ldots+e_n\rangle=3k-1,$$
where $n=2k$. Thus, $\mathrm{Ann}(\mathscr R)$ is not isotropic.

Since codimension of $\mathscr R$ is two it is a maximal poor restriction space.
\end{exam}

\section{Instability of minimal interval exchange transformations with rich restrictions}\label{instability}

Suppose we have a minimal interval exchange transformation $T_{\pi,\mathbf a}$ such that
the full restriction space $\mathscr R(T_{\pi,\mathbf a})$ is rich. Assume for simplicity that $T_{\pi,\mathbf a}$ is uniquely ergodic.
What happens under small perturbation of the parameters?

In this case, as we have seen above, the asymptotic cycle of $T_{\pi,\mathbf a}$ lies in $\mathscr R(T_{\pi,\mathbf a})$.
Therefore, the cycles $c_{\pi,\mathbf a,k}(x)$ form a smaller and smaller angle with $T_{\pi,\mathbf a}$ when $k$ grows.
If we take $\mathbf a'$ close to $\mathbf a$ and satisfying the same restrictions, for large enough $k$, some new cycles appear among
$c_{\pi,\mathbf a',k}(x)$, which do not belong to $\mathcal C_{\pi,\mathbf a,k}$ but are linear combinations with natural coefficients
of cycles from $\mathcal C_{\pi,\mathbf a,j}$ with $j<k$.

These new cycles still form a small angle with $\mathscr R(T_{\pi,\mathbf a})$
and there is a good chance that some of them will have the form $\widehat r$ with $r\in\mathscr R(T_{\pi,\mathbf a})$.
This means that some new restriction that was not realized by~$T_{\pi,\mathbf a}$, will be realized by~$T_{\pi,\mathbf a'}$.
The new realized restrictions may cause $T_{\pi,\mathbf a'}$ to be non-minimal. And if this happens,
it remains true for $T_{\pi,\mathbf a''}$ if $\mathbf a''$ is sufficiently close to $\mathbf a'$ and satisfies the same restrictions.

So, it is natural to expect that minimal interval exchange transformations $T_{\pi,\mathbf a}$ will never
be $\mathscr R(T_{\pi,\mathbf a})$-stably minimal.

\begin{conj}\label{unstable-conj}
If $\mathscr R$ is a rich restriction space with respect to $\langle\,,\rangle_\pi$ and $T_{\pi,\mathbf a}$
satisfies all restrictions from~$\mathscr R$, then $T_{\pi,\mathbf a}$ is not $\mathscr R$-stably minimal.
\end{conj}

Due to stability of non-minimality (Proposition~\ref{zorich-general}) this conjecture immediately implies
the following.

\begin{coro}[to Conjecture~\ref{unstable-conj}]
Let $\mathscr R$ be a rich restriction space with respect to $\langle\,,\rangle_\pi$.
Then the subset $M_{\pi,\mathscr R}$ of minimal exchange
transformations from $X_{\pi,\mathscr R}=\{T_{\pi,\mathbf a}\;;\;\mathscr R(T_{\pi,\mathbf a})\supset\mathscr R\}$
is nowhere dense in~$X_{\pi,\mathscr R}$.
\end{coro}

In some cases, $M_{\pi,\mathscr R}$ is contained in a codimension one subset
of $X_{\pi,\mathscr R}$, and in some other cases, the codimension of
$M_{\pi,\mathscr R}$ in $X_{\pi,\mathscr R}$ is known to be between $0$ and $1$, see Example~\ref{rauzy} and Subsection~\ref{spi}.
This motivates us to conjecture that the following general statement is true.

\begin{conj}\label{measurezero}
Let $\mathscr R$ be a rich restriction space with respect to $\langle\,,\rangle_\pi$.
Then the subset $M_{\pi,\mathscr R}$ has zero Lebesgue measure in $X_{\pi,\mathscr R}$.
\end{conj}

\section{Examples}\label{Examples}
\subsection{\newtext{Rank two interval exchange transformations} }\label{ranktwo}
\newtext{\emph{The rank} of an interval exchange
transformation $T_{\pi,\mathbf a}$ is defined as $\rank_{\mathbb Q}\langle a_1,\ldots,a_n\rangle$.
M.\,Boshernitzan~\cite{b88} shows---in different terms---that
a rank two minimal interval echange transformations is always
stably minimal and uniquely ergodic. This can be viewed as a simple
illustration to the phenomenon discussed
in this paper as we have the following.

\begin{prop}
If $T_{\pi,\mathbf a}$ is a rank two minimal interval exchange transformation,
then $\mathscr R(T_{\pi,\mathbf a})$ is poor.
\end{prop}
\begin{proof}
If $T_{\pi,\mathbf a}$ has rank two, then $\omega_{\pi,\mathbf a}$ has the form
$f^*(\lambda_1\,d\varphi_1,+\lambda_2\,d\varphi_2)$, where
$f:\Sigma_{\pi,\mathbf a}\rightarrow\mathbb T^2$ is a smooth map,
$\varphi_1,\varphi_2$ are angular coordinates on the torus $\mathbb T^2$,
and $\lambda_1,\lambda_2$ are incommensurable reals. The large
freedom in choosing the map $f$ can be used to make the number
of preimages of any point of $\mathbb T^2$ uniformly bounded from above by some $k\in\mathbb N$.

Let $\gamma$ be a regular leaf of $\mathcal F_{\pi,\mathbf a}$. Its image $f(\gamma)$
in $\mathbb T^2$ is contained in a straight line $\ell$, which is a leaf of an irrational winding of $\mathbb T^2$.
If $T_{\pi,\mathbf a}$ is minimal, then $\gamma$ is not closed.
Since $f(\gamma)$ visits every point of $\ell$ at most $k$ times it follows
that for a certain parametrization $\gamma(t)$ there must be a non-zero limit
$$\frac1t\int_0^tdf\bigl(\gamma(t)\bigr)\quad(t\rightarrow\infty).$$
This means that the image of the asymptotic cycle $c_{\pi,\mathbf a}$ under
the induced map $f_*:H_1(\Sigma_{\pi,\mathbf a})\rightarrow H_1(\mathbb T^2)$
is not zero. The kernel of this map is $\mathscr R(T_{\pi,\mathbf a})\cap H_1(\Sigma_{\pi,\mathbf a})$,
so we have $c_{\pi,\mathbf a}\notin\mathscr R(T_{\pi,\mathbf a})$.
By Proposition~\ref{obvious} $\mathscr R(T_{\pi,\mathbf a})$ is poor.
\end{proof}
}
\subsection{Quadratic differentials}
It is well known that the measured singular foliation defined on a Riemann surface by the real part of
a generic quadratic differential with prescribed multiplicities of zeros
is minimal. H.\,Masur even shows in~\cite{m82} that almost all such foliations are uniquely
ergodic.

Let $M$ be an oriented surface.
The family of all measured foliations on $M$ that can be defined by a quadratic differential $q$ on $M$ (for some complex structure)
with zeros at fixed points with prescribed multiplicities gives rise to a family of foliations on the double cover $\widehat M$
branched at zeros of $q$ of odd multiplicity. A foliation from this family can be defined by a closed $1$-form on $\widehat M$
but this $1$-form is already not generic even if so is $q$.

More precisely, the family of $1$-forms $\omega$ that
we obtain in this way is characterized locally by the set~$S$ of zeros with fixed multiplicities  and
the condition $\iota^*\omega=-\omega$, where $\iota$ is the involution of $\widehat M$ that exchanges the sheets
of the covering map $\widehat M\rightarrow M$.

Thus, we have a family of closed $1$-forms with restrictions
$$\int\limits_{c+\iota_*c}\omega=0,$$
where $c\in H_1(\widehat M,S;\mathbb Z)$. The restriction space here is 
\begin{equation}\label{Rqd}
\mathscr R=\{c\in H_1(\widehat M,S;\mathbb R)\;;\;c=\iota_*c\}
\end{equation}

\begin{prop}
Let $\iota:\widehat M\rightarrow\widehat M$ be an orientation preserving involution. Then
the restriction space~\eqref{Rqd} is poor.
\end{prop}

\begin{proof}
There exists  a non-zero cycle $c\in H_1(M;\mathbb R)$ such that $\iota_*c=-c$.
Since $\iota$ preserves the orientation of $M$ it also preserves
the intersection index. So, for any $c'\in\mathscr R\cap H_1(M;\mathbb R)$
we have $c\frown c'=\iota_*(c)\frown\iota_*(c')=-c\frown c'$, hence $c\frown c'=0$.
On the other hand, $c\notin\mathscr R$.
Thus $\mathscr R\cap H_1(M;\mathbb R)$ is not coisotropic.
\end{proof}

Masur's result on the unique ergodicity for quadratic differentials~\cite{m82} supports our Conjecture~\ref{conj1} in the part
related to Theorem~\ref{uniquelyergodictheo} in this case.

\subsection{Interval exchange maps with flips}
A.\,Nogueira proved in \cite{no89} that almost all interval exchange transformations with flips have periodic orbits.
The suspension surface $M$ of an interval exchange transformation with flips
is defined similarly to the case of ordinary interval exchange transformations, but now it is non-orientable.
The foliation $\mathcal F$ induced on $M$ has orientable leaves but it is not \emph{co}orientable. Let $\widehat M$ be
the orientation double cover of $M$. The preimage of $\mathcal F$ on $\widehat M$ can be defined by a closed
$1$-form~$\omega$.

Let $\iota$ be the involution of $\widehat M$ that exchanges the sheets of the covering map $\widehat M\rightarrow M$. Similarly to
the previous case we will have $\iota^*\omega=-\omega$, which gives rise to the same restriction space~\eqref{Rqd}.
However, now $\iota$ flips the orientation of $\widehat M$, which inverses the conclusion.

\begin{prop}
Let $\iota:\widehat M\rightarrow\widehat M$ be an orientation reversing involution. Then
the restriction space~\eqref{Rqd} is rich.
\end{prop}

\begin{proof}
The space $H_1(\widehat M,\mathbb R)$ is symplectic, and the involution $\iota_*$ restricted to it changes
the sign of the symplectic form. This implies that $H_1(\widehat M,\mathbb R)$ splits into a direct sum
of $\pm1$-eigenspaces, which are both Lagrangian. The $-1$-eigenspace is $\mathscr R$.
\end{proof}

\subsection{Measured foliations on non-orientable surfaces}
C.\,Danthony and A.\,Nogueira proved in \cite{dn90} that almost all measured foliations on non-orientable surfaces have a compact leaf.

Let $M$ be a non-orientable surface and $\mathcal F$ a measured singular foliation on $M$. Let $\widehat M$
be the double brunched cover of $M$ on which the preimage of $\mathcal F$ becomes orientable, and $\iota$
the corresponding involution of $\widehat M$.

If $\widehat M$ is orientable then the preimage of $\mathcal F$ on $\widehat M$ is coorientable,
and hence, can be defined by a closed $1$-form $\omega$. We come to exactly the same situation as in
the previous example: $\iota^*\omega=-\omega$, $\iota_*[M]=-[M]$. Thus, foliations close to $\mathcal F$
(with the same singularities) give rise to a family of $1$-forms satisfying rich restrictions.

If $\widehat M$ is non-orientable, then the first return map of $\mathcal F$ on a suitable transversal
is an interval exchange map with flips \emph{and restrictions}. So, after proceeding to the orientation
double cover of $\widehat M$ measured foliations close to $\mathcal F$
will translate to a family of $1$-forms satisfying all restrictions from~\eqref{Rqd} and some additional restrictions.
But the restriction space~\eqref{Rqd} is already rich, so the larger restriction space will also be rich.

\subsection{Novikov's problem}
Let $M$ be a closed and homologous to zero
surface smoothly embedded in the 3-torus $\mathbb T^3$ and $H\in\mathbb R^3$
a non-zero vector. Denote by $\eta$ the following $1$-form on $\mathbb T^3$: $\eta=H_1\,dx_1+
H_2\,dx_2+H_3\,dx_3$, where $x_i$ are the angular coordinates (defined up to $2\pi k$ with $k\in\mathbb Z$)
on $\mathbb T^3$, and by $\omega$ the restriction $\eta|_M$.

The closed $1$-form $\omega$ defines a foliation on $M$ whose leaves lifted to $\mathbb R^3$
may or may not be open and have an asymptotic direction. In 1982 S.\,Novikov suggested
to study their behavior in connection with the theory of conductivity in normal metals~\cite{n82}.

The forms $\omega$ that arise in this situation are not generic since $\int_c\omega=0$ whenever
$c\in H_1(M,\mathbb Z)$ has zero image in $H_1(\mathbb T^3,\mathbb Z)$ under
the embedding $M\hookrightarrow\mathbb T^3$. Locally these are the only restrictions,
so, the problem can be reduced to studying interval exchange transformations
with the restriction space $\mathscr R(i)$ generated by all cycles $c\in H_1(\Sigma,\mathbb Z)$
such that $i_*(c)=0$, where $\Sigma$ is the suspension surface and $i:\Sigma\hookrightarrow\mathbb T^3$
is some embedding.

Note, however, that it is not clear whether all interval exchange transformations obtained in this way
can come from the restriction of a $1$-form on $\mathbb T^3$ with \emph{constant} coefficients
to a surface. So, the family of transformations that can arise in Novikov's problem is probably a
proper part of
the family of all interval exchange transformations with restrictions of the form $\mathscr R(i)$,
but it has the same number of degrees of freedom.

\begin{prop}
If $i_*([\Sigma])=0\in H_2(\mathbb T^3,\mathbb Z)$, then the restriction space $\mathscr R(i)$
is rich.
\end{prop}

\begin{proof}
Let $\omega_1$ and $\omega_2$ be two closed $1$-forms on $\Sigma$ whose homology
classes satisfy all restrictions from~$\mathscr R(i)$. Then there are closed $1$-forms $\eta_1$, $\eta_2$
on $\mathbb T^3$ such that $\omega_j=i^*\eta_j$, $j=1,2$. We have
$$\langle\omega_1,\omega_2\rangle=\int_\Sigma\omega_1\wedge\omega_2=
\int_\Sigma i^*(\eta_1\wedge\eta_2)=\int_\Sigma\eta_1\wedge\eta_2=0$$
since $i_*[\Sigma]=0$.

We see that $\mathrm{Ann}(\mathscr R(i))$ is isotropic, so $\mathscr R(i)$ is rich. 
\end{proof}

A.\,Zorich's result~\cite{z84} can be interpreted in terms of interval exchange transformations
as follows: minimal foliations arising in Novikov's problem are never $\mathscr R(i)$-stable.
It was shown later in~\cite{d93} that almost all foliations in Novikov's problem
are \emph{not} minimal. Moreover, all minimal foliations reside in a codimension one subset $X$
of the set of all relevant foliations.

The subset $X$ is defined (locally) by another restriction, and if we add this
restriction to the corresponding $\mathscr R(i)$ the question of minimality becomes highly nontrivial.
It is conjectured (in different terms) that almost all foliations from $X$ are \emph{not}
minimal \cite{MN}.

In the genus $3$ case, a \remove{two-parametric}
\newtext{two parameter} subfamily of $X$ is studied in \cite{dd08} and it is shown that
minimal foliations in $X$ form a subset known as the Rauzy gasket.

\subsection{Systems of partial isometries}\label{spi}
\begin{defi}
By \emph{a system of partial isometries} we mean a collection $\Psi=\{\psi_1,\ldots,\psi_k\}$
of interval isometries $\psi_i:[a_i,b_i]\rightarrow[c_i,d_i]$, where $[a_i,b_i]$ and $[c_i,d_i]$
are subintervals of $[0,1]$ of equal length for each $i=1,\ldots,k$.
It is required that $\{a_1,c_1,\ldots,a_k,c_k\}$ contains $0$ and
$\{b_1,d_1,\ldots,b_k,d_k\}$ contains $1$. The isometries $\psi_i$
can preserve or reverse orientation. The system is called \emph{orientation preserving}
if all $\psi_i$'s preserve orientation.

Two systems of partial isometries obtained from each other by permuting $\psi_i$'s and
replacing some $\psi_i$ by $\psi_i^{-1}$ are considered equal.

\emph{The $\Psi$-orbit of a point $x\in[0,1]$} is the set
of images of $x$ under all compositions $\psi_{i_1}^{\pm1}\circ\ldots\circ\psi_{i_m}^{\pm1}$,
$m=0,1,2,\ldots$, that are defined at $x$.
\end{defi}

With every system of partial isometries $\Psi$ one associates a foliated $2$-dimensional complex called
\emph{the suspension complex of $\Psi$} and denoted here $\Sigma_\Psi$, see~\cite{glp94}.
Such foliated complexes are particular cases of \emph{band complexes} that
were studied in connection with geometric group theory and
the theory of $\mathbb R$-trees~\cite{bf95}, \cite{glp94}.

Minimal band complexes (the notion of minimality is slightly
more subtle here but for $\Sigma_\Psi$, outside of a codimension one subset of such complexes, it is equivalent
to saying that orbits of $\Psi$ are everywhere dense) are classified by E.\,Rips into toral type, surface type,
and thin type~\cite{bf95}.

Oriented systems of partial isometries can be viewed as a generalization of interval exchange transformations
so that interval exchange transformations will appear as systems of partial isometries satisfying certain integral linear restrictions
on the parameters.
There is, however, another relation between these objects that
associates an interval exchange transformation satisfying a number of restrictions
to any orientation preserving system of partial isometries.

\begin{defi}\label{doublesusp}
Let $\Psi=\{\psi_i:[a_i,b_i]\rightarrow[c_i,d_i]\;;\;i=1,\ldots,k\}$ be an orientation preserving system of partial isometries and
$\mathbf y=(y_1,y_2,\ldots,y_{2k})$ a $2k$-tuple of pairwise distinct points of $(0,1)$. Such a $2k$-tuple
will be referred to as \emph{admissible for $\Psi$}.
The surface $\widehat\Sigma_{\Psi,\mathbf y}$ that is constructed below will be called
\emph{a double suspension surface of $\Psi$}.

Choose $\varepsilon>0$ so small that the intervals $[y_i,y_i+\varepsilon]$ are pairwise disjoint and contained in $[0,1)$.
The surface $\widehat\Sigma_{\Psi,\mathbf y}$ is obtained from
$$D=[0,1]\times[0,1]\setminus\bigcup_{i=1}^k\bigl((a_i,b_i)\times(y_i,y_i+\varepsilon)\cup(c_i,d_i)\times(y_{i+k},y_{i+k}+\varepsilon)\bigr)$$
by making the following, orientation and measure preserving, identifications:
$$\begin{aligned}\relax
[0,1]\times\{0\}&\text{ with }[0,1]\times\{1\},\\
[a_i,b_i]\times\{y_i\}&\text{ with }[c_i,d_i]\times\{y_{i+k}+\varepsilon\},\\
[a_i,b_i]\times\{y_i+\varepsilon\}&\text{ with }[c_i,d_i]\times\{y_{i+k}\}
\end{aligned}$$
and collapsing to a point every straight line segment in the boundary of the domain~$D$, namely,
$\{0\}\times[0,1]$, $\{1\}\times[0,1]$, $a_i\times[y_i,y_i+\varepsilon]$, $b_i\times[y_i,y_i+\varepsilon]$,
$c_i\times[y_{i+k},y_{i+k}+\varepsilon]$, and $d_i\times[y_{i+k}+\varepsilon]$, $i=1,\ldots,k$.
\end{defi}

The surface $\widehat\Sigma_{\Psi,\mathbf y}$ comes with a closed differential $1$-form $\omega$
whose pullback in $D$ is $dx$ (a smooth structure on $\widehat\Sigma_{\Psi,\mathbf y}$ can
be chosen so that $\omega$ is smooth). The $1$-form $\omega$ defines a measured
orientable foliation $\widehat{\mathcal F}_{\Psi,\mathbf y}$,
whose first return map on the transversal $\gamma=[0,1]\times\{0\}\sim[0,1]\times\{1\}$ is an interval exchange.
We denote this transformation by $T_{\Psi,\mathbf y}$.

It is not, however, necessarily true that all leaves of the foliation~$\widehat{\mathcal F}_{\Psi,\mathbf y}$ meet $\gamma$.
If it is we say that the transformation~$T_{\Psi,\mathbf y}$ \emph{fills} the surface $\widehat\Sigma_{\Psi,\mathbf y}$.
If it is not, then the foliation~$\widehat{\mathcal F}_{\Psi,\mathbf y}$ is not minimal.

If $T_{\Psi,\mathbf y}$ fills~$\widehat\Sigma_{\Psi,\mathbf y}$, then the minimality of~$\widehat{\mathcal F}_{\Psi,\mathbf y}$ is equivalent to
that of $T_{\Psi,\mathbf y}$.
Indeed, in this case, the foliated
surface~$\widehat\Sigma_{\Psi,\mathbf y}$
can be identified with the suspension surface of $T_{\Psi,\mathbf y}$, possibly after collapsing to a point
some saddle connections. Such collapsing will not be essential in the sequel, so we will
think of~$\widehat\Sigma_{\Psi,\mathbf y}$ as the suspension surface for~$T_{\Psi,\mathbf y}$ provided
that the latter fills~$\widehat\Sigma_{\Psi,\mathbf y}$.

\begin{prop}\label{spi-rich}
If $T_{\Psi,\mathbf y}$ fills $\widehat\Sigma_{\Psi,\mathbf y}$, then
the full restriction space $\mathscr R(T_{\Psi,\mathbf y})$ is rich.
\end{prop}

\begin{proof}
The point is that $\widehat\Sigma_{\Psi,\mathbf y}$ can be realized as the boundary
of a handlebody $H_{\Psi,\mathbf y}$ such that $\omega$ is continued to
$H_{\Psi,\mathbf y}$ as a closed $1$-form. The handlebody is obtained from $[0,1]^3$
by the following identifications:
\begin{enumerate}
\item each of the two squares $\{0\}\times[0,1]\times[0,1]$ and $\{1\}\times[0,1]\times[0,1]$
is collapsed to a point;
\item for any $x\in(0,1)$ the straight line segment $[(x,0,0),(x,1,0)]$ is collapsed to a point;
\item the square $[0,1]\times\{0\}\times[0,1]$ is identified with $[0,1]\times\{1\}\times[0,1]$
by the projection along the second axis;
\item for each $i=1,\ldots,k$ the rectangle $[a_i,b_i]\times[y_i,y_i+\varepsilon]\times\{1\}$
is identified with $[c_i,d_i]\times[y_{i+k},y_{i+k}+\varepsilon]\times\{1\}$ by the map
$(x,y,1)\mapsto(x+c_i-a_i,y_{i+k}+\varepsilon-y_i-y,1)$;
\item
collapse each straight line segment $\{a_i\}\times[y_i,y_i+\varepsilon]\times\{1\}$ and $\{b_i\}\times[y_i,y_i+\varepsilon]\times\{1\}$
to a point, $i=1,\ldots,k$.
\end{enumerate}
One can see that this produces a handlebody whose boundary is obtained from $D\times\{1\}$ by
the same identifications as in Definition~\ref{doublesusp}.

Now, it is a general fact that any closed $1$-form on a handlebody $H$ restricted to $\partial H$ satisfies
a rich set of restrictions, which always contains the kernel of the map $H_1(\partial H,\mathbb Z)\rightarrow H_1(H,\mathbb Z)$
induced by the inclusion $\partial H\hookrightarrow H$. Indeed, for any two such forms $\omega_1$, $\omega_2$ the product
$\omega_1\wedge\omega_2$ is exact as $H^2(H)=0$. Therefore, $\int_{\partial H}\omega_1\wedge\omega_2=0$.
\end{proof}

So, according to Conjecture~\ref{measurezero} we expect that interval exchange transformations coming from systems
of partial isometries and the double suspension surface construction will typically be non-minimal if they
happen to fill the double suspension surface.

With every band complex $\Sigma$ Bestvina and Feighn associate a number $e(\Sigma)$ called \emph{the excess of $\Sigma$},
see~\cite{bf95} for the definition. For any $\Sigma$, the Rips machine generates a sequence $\Sigma_0=\Sigma,\Sigma_1,\Sigma_2,\ldots$
of band complexes such that $e(\Sigma_{i+1})\leqslant e(\Sigma_i)$ for all $i$. Moreover, for some $n$, we have
$e(\Sigma_n)=e(\Sigma_{n+i})$ for all~$i\geqslant0$. We call $e(\Sigma_n)$ \emph{the asymptotic excess of $\Sigma$}.

In the case when $\Psi=\{\psi_i:[a_i,b_i]\rightarrow[c_i,d_i]\;;\;i=1,\ldots,k\}$ is an orientation preserving system of
partial isometries such that the differences $c_i-a_i$, $i=1,\ldots k$, are independent over $\mathbb Q$
the excess of the suspension complex is simply
$$e(\Sigma_\Psi)=\sum_{i=1}^k(b_i-a_i)-1,$$
and it coincides with the asymptotic excess of $\Sigma_\Psi$.

In any case, the asymptotic excess of $\Sigma_\Psi$ is equal to a non-trivial integral linear combination
of the parameters $a_i,b_i,c_i,d_i$, $i=1,\ldots,k$ (by definition one of $a_i$'s is zero,
so it is understood that this $a_i$ is excluded from consideration when we speak about
dependencies and linear combinations of the parameters). So, if the parameters are rationally independent
the asymptotic excess is not zero. We think that this implies non-minimality of $\widehat{\mathcal F}_{\Psi,\mathbf y}$
for any $\mathbf y$.

\begin{conj}
Let $\Psi$ be an orientation preserving system of partial isometries.
If $\Sigma_\Psi$ has non-zero asymptotic excess, then, for any
admissible $\mathbf y$, the corresponding
interval exchange transformation $T_{\Psi,\mathbf y}$ either is non-minimal or
does not fill~$\widehat\Sigma_{\Psi,\mathbf y}$.
\end{conj}

One half of this conjecture is easy: if the asymptotic excess of $\Sigma_\Psi$ is negative,
then $\Sigma_\Psi$ has compact leaves (equivalently, $\Psi$ has finite orbits), which immediately
implies that $\widehat{\mathcal F}_{\Psi,\mathbf y}$ has compact leaves for any admissible $\mathbf y$, and thus, is not minimal.

If the asymptotic excess is positive, then $\Sigma_\Psi$ is either not minimal (and then
the conjecture also follows easily) or of toral type. So, it remains to establish
the conjecture in the toral case. So far we can do this in the case $k=3$ for a special choice of $\mathbf y$.

\begin{prop}
Let $\Psi=\{\psi_i:[a_i,b_i]\rightarrow[c_i,d_i]\;;\;i=1,2,3\}$ be an orientation preserving system of
partial isometries and let $y_1<y_2<\ldots<y_6$. Then the
inequality $\sum_{i=1}^k(b_i-a_i)>1$ (which holds whenever the asymptotic excess of $\Sigma_\Psi$
is positive) implies that~$\widehat{\mathcal F}_{\Psi,\mathbf y}$ is not minimal.
\end{prop}

A proof of this fact
can be extracted from \cite{d08} where $\Psi$ and the double suspension surface
are realized in the $3$-torus so that the foliations on both are induced by plane sections
of a fixed direction.

\begin{defi}
A system of partial isometries $\Psi=\{\psi_i:[a_i,b_i]\rightarrow[c_i,d_i]\;;\;i=1,\ldots,k\}$  will be called \emph{balanced} if
$\sum_{i=1}^k(b_i-a_i)=1$.
\end{defi}

\begin{prop}\label{balanced-thin}
Let $\Psi$ be a balanced and orientation preserving system of partial isometries, and $\mathbf y$ an admissible vector for $\Psi$.
If, in addition, $\Psi$ is of thin type, then
the transformation $T_{\Psi,\mathbf y}$ is minimal and fills $\widehat\Sigma_{\Psi,\mathbf y}$.
\end{prop}

\begin{proof}
We give only a sketch leaving details to the reader, who is assumed being familiar with
the theory of band complexes and the Rips machine.

As noted above, $T_{\Psi,\mathbf y}$ is minimal and fills $\widehat\Sigma_{\Psi,\mathbf y}$ if and only if
$\widehat{\mathcal F}_{\Psi,\mathbf y}$ is minimal.

There is a leafwise embedding of the suspension complex $\Sigma_\Psi$ into the foliated
handlebody $H_{\Psi,\mathbf y}$ introduced in the proof of Proposition~\ref{spi-rich}
such that, for any leaf $L$ of $\Sigma_\Psi$, the embedding of $L$
into the corresponding leaf of $H_{\Psi,\mathbf y}$ is a quasi-isometry and a homotopy equivalence.
The image of $\Sigma_\Psi$ in $H_{\Psi,\mathbf y}$ is the union of the straight line segment $[0,1]\times\{0\}\times\{0\}$
and the rectangles
$$\bigcup_{i=1}^k([a_i,b_i]\times\{y_i+\varepsilon/2\}\times[0,1]\cup[c_i,d_i]\times\{y_{i+k}+\varepsilon/2\}\times[0,1])$$
under the identifications from the proof of Proposition~\ref{spi-rich}.

Since $\Psi$ is balanced, the presence of compact leaves in $\Sigma_\Psi$ is equivalent
to the presence of regular leaves that are not simply connected. So, if $\Psi$ is minimal,
then all regular leaves of $\Sigma_\Psi$ are infinite trees.

If $\Psi$ is of thin type, then there is an everywhere dense leaf $L$ of $\Sigma_\Psi$ having form of a $1$-ended
infinite tree. The corresponding leaf of $H_{\Psi,\mathbf y}$ has connected boundary, which is everywhere dense
in $\partial H_{\Psi,\mathbf y}=\widehat\Sigma_{\Psi,\mathbf y}$. Therefore,
the foliation in $\widehat\Sigma_{\Psi,\mathbf y}$, and hence, the foliation~$\widehat{\mathcal F}_{\Psi,\mathbf y}$,
is minimal.
\end{proof}

Propositions~\ref{spi-rich} and \ref{balanced-thin} together with Conjecture~\ref{unstable-conj} (if proven)
imply that being of thin type is an unstable property of a balanced system of partial isometries. We will
prove that instability does occur in the case when
the balancedness is the only linear restriction on the parameters of $\Psi$.

\begin{prop}
Let $\Psi$ be a balanced and orientation preserving system of partial isometries with parameters that do
not satisfy any linear integral restriction except balancedness.
Let $\mathbf y$ be an admissible vector such that $T_{\Psi,\mathbf y}$
fills the double suspension surface $\widehat\Sigma_{\Psi,\mathbf y}$.
Then $T_{\Psi,\mathbf y}$ is not $\mathscr R(T_{\Psi,\mathbf y})$-stably
minimal.
\end{prop}

\begin{proof}
For any open neighborhood $U$ of $\Psi$ one can find a
system of partial isometries $\Psi'=\{\psi_i':[a_i',b_i']\rightarrow[c_i',d_i']\;;\;i=1,\ldots,k\}$ in $U$
such that 
\begin{enumerate}
\item
the shifts $c_i'-a_i'$ are pairwise commensurable;
\item
no integral linear relation for the parameters of $\Psi'$ holds that is
not a consequence of this commensurability and the balancedness of $\Psi'$.
\end{enumerate}
Then the $1$-form $\omega=dx$ on the handlebody $H_{\Psi',\mathbf y}$ has pairwise
commensurable periods over all integral cycles. Therefore, it is proportional to the pullback of the
angular form $d\varphi$ on the circle $S^1$ under a map $\theta:H_{\Psi',\mathbf y}\rightarrow S^1$.
The connected components of the sets $\theta^{-1}(x)$, $x\in S^1$,
are the leaves of the foliation~$\mathcal F_{\Psi',\mathbf y}$ induced by $\omega$ on $H_{\Psi',\mathbf y}$.

The absence of additional relations on the parameters implies that
any subset $\theta^{-1}(x)$ contains only one singularity of $\mathcal F_{\Psi',\mathbf y}$.
Therefore, the Euler characteristics $\chi(\theta^{-1}(x))$, which is a piecewise constant function of $x$,
can jump only by one when $x$ varies continuously.
On the other hand,
the balancedness of $\Psi'$ implies that $\chi(\theta^{-1}(x))$ averages to zero
on $S^1$. Therefore, for some $x\in S^1$, we must have $\chi(\theta^{-1}(x))>0$, which
means that at least one of the connected components of $\theta^{-1}(x)$ is a
disc. The boundary $\rho$ of this disc is a realization of some restriction $r$ that holds for
$T_{\Psi,\mathbf y}$ since $\rho$ is homologous to zero in $H_{\Psi,\mathbf y}$. For
any $\Psi''$ close enough to $\Psi'$ the restriction $r$ will be realized also
by $T_{\Psi'',\mathbf y}$.
\end{proof}

Now we consider a particular family of systems of partial isometries with $k=3$ in which
the subset of systems giving rise to minimal interval exchange transformations is
precisely known, and explain the origin of Example~\ref{rauzy}.

Let $\Psi$ be an orientation preserving system of partial isometries of the form
$$\{\psi_i:[0,b_i]\rightarrow[c_i,1],\ i=1,2,3\},\quad b_1<b_2<b_3<c_2,\quad c_3<b_3,\quad 2b_3<b_2+c_1,\quad b_i+c_i=1,$$
and $\mathbf y\in(0,1)^6$ be an arbitrary vector with $y_1<y_2<\ldots<y_6$.
The domain $D$ and identifications in it to obtain $\widehat\Sigma_{\Psi,\mathbf y}$ are shown in Fig.~\ref{domain}.
\begin{figure}
\centerline{\includegraphics{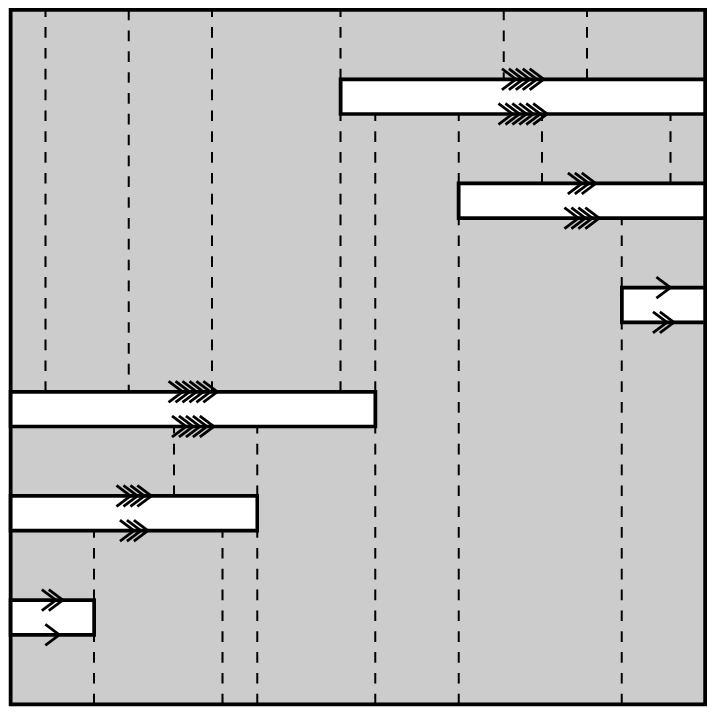}\put(-200,15){$1$}\put(-170,15){$2$}\put(-147,15){$3$}\put(-125,15){$4$}%
\put(-95,15){$5$}\put(-60,15){$6$}\put(-25,15){$7$}\put(-200,47){$7$}\put(-190,77){$6$}\put(-154,77){$1$}\put(-113,137){$3$}%
\put(-70,167){$7$}\put(-40,167){$2$}\put(-18,167){$3$}\put(-25,137){$1$}\put(-208,200){$3$}%
\put(-191,200){$5$}\put(-167,200){$7$}\put(-136,200){$2$}\put(-93,200){$6$}\put(-58,200){$1$}\put(-30,200){$4$}\hskip1cm
\includegraphics{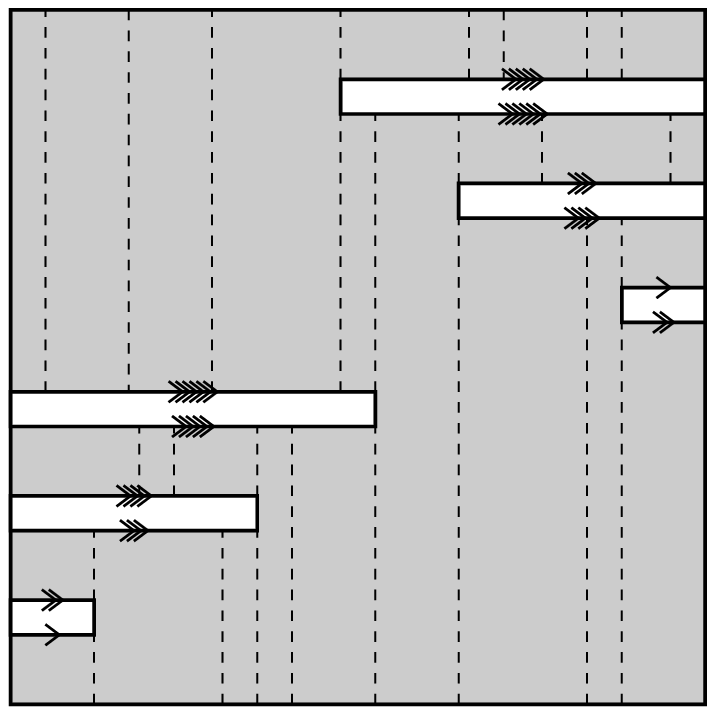}\put(-200,15){$1$}\put(-170,15){$2$}\put(-147,15){$3$}%
\put(-137,15){$4$}\put(-120,15){$5$}\put(-41,15){$4$}%
\put(-95,15){$6$}\put(-63,15){$7$}\put(-25,15){$5$}\put(-200,47){$5$}\put(-154,77){$1$}\put(-113,137){$3$}%
\put(-193,77){$7$}\put(-171,77){$4$}%
\put(-70,167){$5$}\put(-40,167){$2$}\put(-18,167){$3$}\put(-25,137){$1$}\put(-208,200){$3$}%
\put(-191,200){$6$}\put(-167,200){$5$}\put(-136,200){$2$}\put(-58,200){$1$}\put(-41,200){$4$}%
\put(-25,200){$5$}\put(-75,200){$4$}\put(-97,200){$7$}}
\caption{Double suspension surface. Each vertical straight line segment
in $\partial D$ is collapsed to a point, the bottom side of the square is identified with the top one}\label{domain}
\end{figure}

One can see from the left picture that the interval exchange transformation $T_{\Psi,\mathbf y}$ fills
the double suspension surface and is associated with
the following permutation:
$$\pi=\begin{pmatrix}1&2&3&4&5&6&7\\3&5&7&2&6&1&4\end{pmatrix}$$
and the following vector of parameters:`
$$\mathbf a=(b_1,b_2+c_1-2b_3,b_3-c_3,b_3-b_2,c_2-b_3,b_2-b_1,b_1).$$
The dashed lines in the left picture of
Fig.~\ref{domain} show the segments of separatrices between the singularity and the first intersection with $[0,1]\times\{0\}$.
The corresponding foliation on the double suspension surface has two singularities each of which is a double saddle.

If the excess $e(\Psi)=b_1+b_2+b_3-1$ is not equal to zero, then $T_{\Psi,\mathbf y}$ is not minimal. In the case
$e(\Psi)=0$ the double suspension surface can be identified with the quotient of the regular skew polyhedron $\{4,6\;|\;4\}$ by~$\mathbb Z^3$,
embedded in the $3$-torus with the foliation induced by the $1$-form $(b_2+b_3)\,dx_1+(b_3+b_1)\,dx_2+(b_1+b_2)\,dx_3$,
see~\cite{dd08,dynskr}. The minimality of $T_{\Psi,\mathbf y}$ is equivalent to $\Psi$ being of thin type.
This occurs if and only if the point $(b_1:b_2:b_3)\in\mathbb RP^2$ belongs to the Rauzy gasket, which is defined as follows.

Denote by $\Delta$ the following subset of the projective plane $\mathbb RP^2$: $\Delta=\{(x_1:x_2:x_3)\;;\;x_i\geqslant0,\ i=1,2,3\}$,
and by $P_1$, $P_2$, $P_3$, the projective transformations with matrices
$$\begin{pmatrix}1&1&1\\0&1&0\\0&0&1\end{pmatrix},\quad
\begin{pmatrix}1&0&0\\1&1&1\\0&0&1\end{pmatrix},\quad
\begin{pmatrix}1&0&0\\0&1&0\\1&1&1\end{pmatrix},$$
respectively.

Points of the Rauzy gasket are in one-to-one correspondence with infinite sequences $(i_1,i_2,\ldots)\in\{1,2,3\}^{\mathbb N}$
in which every element of $\{1,2,3\}$ appears infinitely often. For any such sequence the intersection
$$\bigcap_{m=1}^\infty(P_{i_1}\circ\ldots\circ P_{i_m})(\Delta)$$
consists of a single element, which is the point of the Rauzy gasket corresponding to the sequence.

The interval exchange transformations from Example~\ref{rauzy}
was produced by using the construction above with the additional assumption $b_1+b_2<b_3$.
One can see from the right picture in Fig~\ref{domain} that
the first return map for the shorter transversal $[0,b_2+c_3]$
will be the interval exchange transformation associated with the permutation
$$\begin{pmatrix}1&2&3&4&5&6&7\\3&6&5&2&7&4&1\end{pmatrix}$$
and the following vector of parameters:
$$\frac1{b_2+c_3}(b_1,b_2+c_1-2b_3,b_3-c_3,b_3-b_1-b_2,b_1,c_2-b_3,2b_2-b_3).$$

The parameters $a_i$, $i=1,2,3$, and $e$ from Example~\ref{rauzy} are related to these as follows:
$$a_1=\frac{b_1}{b_2+c_3},\quad a_2=\frac{b_2+c_1-2b_3}{b_2+c_3},\quad a_3=\frac{b_3-c_3}{b_2+c_3},
\quad e=\frac{e(\Psi)}{b_2+c_3}=\frac{b_1+b_2+b_3-1}{b_2+c_3}.$$
The triangle defined by
$$a_i\geqslant0,\quad 3a_1+2a_2+2a_3=1$$
coincides with $P_3(P_2(\Delta))$, so it
intersects the mentioned above Rauzy gasket in a subset that is
also projective equivalent to the Rauzy gasket.

\subsection{Interval translation mappings}
Interval translation mappings were introduced by M.\,Boshernitzan and I.\,Kornfeld in \cite{bk95}.

\begin{defi}
\emph{An interval translation mapping} is a map $T:[0,1)\rightarrow[0,1)$ of the form
$$T(x)=x+t_i\qquad\text{if}\qquad\sum_{j=1}^{i-1}\lambda_j\leqslant x<\sum_{j=1}^i\lambda_j,$$
where $\boldsymbol\lambda=(\lambda_1,\ldots,\lambda_k)$ is a probability vector and $\mathbf t=(t_1,\ldots,t_k)$
is a real vector whose coordinates satisfy the inequalities
$$-\sum_{j=1}^{i-1}\lambda_j\leqslant t_i\leqslant1-\sum_{j=1}^i\lambda_j.$$
The map defined by these parameters will be denoted $T_{\boldsymbol\lambda,\mathbf t}$.

An interval translation mapping $T$ is said to be \emph{of finite type}
if for some $m\in\mathbb N$ the image of $T^{m+1}$ coincides with that of $T^m$,
\newtext{and otherwise \emph{of infinite type}}.
\end{defi}

M.\,Boshernitzan and I.\,Kornfeld ask in \cite{bk95} whether almost all interval translation mappings are of finite type. 
This question is answered in the positive in the particular cases of so called double rotations, see \cite{sia05}, \cite{bc12}, and \cite{bt03}, and for interval translation mappings on three intervals (i.e. in the case $k=3$), see \cite{vo14}.

In general, this problem is an instance of our Conjecture~\ref{measurezero}. Indeed, an interval translation mapping
can be viewed as an orientation preserving system of partial isometries $\Psi=\{\psi_i:[a_i,b_i]\rightarrow[c_i,d_i]\;;\;i=1,\ldots,k\}$ such
that $a_1=0$, $a_i=b_{i-1}$ for $i=2,\dots,k$, $b_k=1$. The only difference is that $\psi_i$'s are defined on the whole
interval $[a_i,b_i]$, including the right endpoint, which is inessential. \remove{`Finite type'}
\newtext{`Infinite type'} for an interval translation mappings means exactly
`thin type' for the corresponding system of partial isometries.

One can show that for this kind of systems of partial isometries the converse statement to Proposition~\ref{balanced-thin}
is also true: the foliation $\widehat{\mathcal F}_{\Psi,\mathbf y}$ is minimal if \emph{and only if} $\Psi$ is of thin type.
So, the problem of studying finite type interval translation mappings translates exactly into
a problem of studying interval exchange maps satisfying certain rich set of integral linear restrictions.

\newtext{Note, however, that the minimality of $\widehat{\mathcal F}_{\Psi,\mathbf y}$ is broken here
in a `non-standard' way. Typically, for a generic system of isometries $\Psi$
the reason for $\widehat{\mathcal F}_{\Psi,\mathbf y}$ to be non-minimal
is the presence of a closed regular leaf, which does not occur for $\Psi$ corresponding to a generic
interval translation mapping. If $\Psi$ comes from a finite type interval translation mapping $T$,
then $\widehat{\mathcal F}_{\Psi,\mathbf y}$ decomposes into two halves having the
same qualitative behavior. In particular, if $T$ is a minimal interval exchange transformation,
then $\widehat{\mathcal F}_{\Psi,\mathbf y}$ will have two minimal components each having
$T$ as the first return map for some transversal.}

\end{document}